\documentclass{amsart}
\usepackage{amssymb,latexsym,amsmath,amscd,graphicx,setspace,amsthm,verbatim,comment}
\usepackage[margin = 3 cm]{geometry} \input xy \xyoption{all}

\theoremstyle{plain}
\newtheorem{The}{Theorem}[section]
\newtheorem{Pro}[The]{Proposition}
\newtheorem{Hyp}[The]{Hypothesis}
\newtheorem{Cor}[The]{Corollary}
\newtheorem{Lem}[The]{Lemma}
\newtheorem{Con}[The]{Conjecture}

\makeatletter
\def\th@remark{%
  \thm@headfont{\bfseries}%
  \normalfont 
  \thm@preskip\topsep \divide\thm@preskip\tw@
  \thm@postskip\thm@preskip
}
\makeatother

\theoremstyle{remark} 
\newtheorem{Rem}[The]{Remark}
\newtheorem{Step}{Step}
      
\theoremstyle{definition}
\newtheorem{Def}[The]{Definition}

\begin{document} 
\title[Numerical evidence for higher order Stark-type conjectures]
{Numerical evidence for higher order Stark-type conjectures}

\author{Kevin McGown, Jonathan Sands, Daniel Valli\`{e}res}



\date{\today}

\begin{abstract}
We give a systematic method of providing numerical evidence for higher order Stark-type conjectures such as (in chronological order) Stark's conjecture over $\mathbb{Q}$, Rubin's conjecture, Popescu's conjecture, and a conjecture due to Burns that constitutes a generalization of Brumer's classical conjecture on annihilation of class groups.  Our approach is general and could be used for any abelian extension of number fields, independent of the signature and type of places (finite or infinite) that split completely in the extension.

We then employ our techniques in the situation where $K$ is a totally real, abelian, ramified cubic extension of a real quadratic field.  We numerically verify the conjectures listed above for all fields $K$ of this type with absolute discriminant less than $10^{12}$, for a total of $19197$ examples.  The places that split completely in these extensions are always taken to be the two real archimedean places of $k$ and we are in a situation where all the $S$-truncated $L$-functions have order of vanishing at least two.  

\end{abstract} 
\maketitle 
\tableofcontents 

\section{Introduction}
In a well-known series of four papers, Harold Stark formulated several conjectures regarding the special value at $s=0$ of  Artin $L$-functions.  In \cite{Stark:1975}, he formulated what is now known as Stark's main conjecture (or Stark's conjecture over $\mathbb{Q}$) for a general  Artin $L$-function, and in \cite{Stark:1980}, he formulated a more refined conjecture for  $L$-functions associated to abelian extensions of number fields having order of vanishing one at $s=0$ (referred to henceforth as Stark's abelian rank one conjecture).  After some previous work of Sands, Stark and Tangedal, Stark's abelian rank one conjecture was extended to higher order of vanishing  $L$-functions by Rubin (Conjecture $B$ of \cite{Rubin:1996}) and by Popescu (Conjecture $C$ of \cite{Popescu:2002}).  Popescu's and Rubin's conjectures are closely related, though not equivalent in general.  Popescu carefully studied a comparison theorem between the two, and he showed that Rubin's conjecture implies his, and at times, they are equivalent.  For more information on these matters, we refer the reader to Theorem $5.5.1$ of \cite{Popescu:2002}.

All these conjectures have been studied extensively by various authors.  An impressive amount of work in gathering numerical evidence for Stark's abelian rank one conjecture has been done over the years.  But to our knowledge, very few authors have provided numerical evidence for Rubin's or Popescu's conjecture in the case where the  $L$-functions have order of vanishing greater than or equal to two.  (The only two such works known to us are \cite{Grant:1999} and \cite{Sands:2001}.)  The goal of this investigation is to remedy this situation.
After completing this paper, it was brought to our attention that Stucky (see~\cite{Stucky:2017}) very recently completed his master's thesis on the subject, but his approach is different than ours.

Roughly speaking, Popescu's conjecture predicts that a certain arithmetical object built out of $S$-units, called an evaluator, lies in a meaningful lattice inside a vector space over $\mathbb{Q}$.  The idea is to use an Artin system of $S$-units in order to give a precise formula for the evaluator that then allows one to check if it lies in the expected lattice.  There is no canonical choice for an Artin system of $S$-units, and different systems give different representations for the evaluator.  Nevertheless, they can be found algorithmically.  It is worth pointing out that Stark originally used Artin systems of $S$-units in order to state his main conjecture in \cite{Stark:1975}, but they have since been superseded by the use of a more abstract result on rationality of linear representations due to Herbrand.  In \S \ref{artinsystem} below, we give a definition of an Artin system of $S$-units, since it is essential to our approach.

It follows from our formula (see Proposition \ref{formulaartin} below)
that the evaluator
will lie in the underlying rational vector space provided Stark's conjecture over $\mathbb{Q}$ is true.  Stark's conjecture over $\mathbb{Q}$ can be interpreted as a rationality statement about an element in $\mathbb{C}[G]$ constructed out of special values of $S$-truncated $L$-functions at $s=0$.  Recently, Burns formulated a conjecture (Conjecture $2.4.1$ in \cite{Burns:2011}) that would provide bounds for the denominators of this element (and also provides a generalization of Brumer's classical conjecture on annihilation of class groups).  Hence, we also give numerical evidence for Stark's conjecture over $\mathbb{Q}$ and Burns's conjecture.  

Also, we note that our work provides numerical evidence for the leading term conjecture (namely, the equivariant Tamagawa number conjecture for the pair $(h^{0}({\rm Spec}(K)),\mathbb{Z}[G])$), since Burns showed in \cite{Burns:2007} that it implies Popescu's conjecture (see also \cite{Burns/Kurihara/Sano:2016}).  Moreover, Burns showed that the leading term conjecture implies his conjecture under some technical conditions (see Theorem $4.1.1$ of \cite{Burns:2011}).  

In this paper, we use our approach to provide numerical evidence for Stark's, Rubin's, Popescu's, and Burns's conjectures by computing the $19197$ examples where the top field is a totally real number field of absolute discriminant less than $10^{12}$ that is a ramified abelian cubic extension of a real quadratic number field and where the split places in the extension are always taken to be the two archimedean ones of the base field (the set $S$ is taken to be the minimal one).  As far as we know, the conjectures above are still open in this setting, except for when the top field is abelian over $\mathbb{Q}$ by previous results of Burns.  (See Theorem $A$ of \cite{Burns:2007} and Corollary $4.1.3$ of \cite{Burns:2011}.)  Our method is fairly general and could be used as well to numerically verify various refinements and generalizations of both Rubin's and Popescu's conjectures, such as Conjecture $4.16$ of \cite{Vallieres:2016} and various other ones contained in \cite{Burns:2011} and \cite{Burns/Boomla:2017}.

Note that for the cubic extensions $K/k$ considered above, the group of roots of unity $\mu(K) = \{\pm 1 \}$ is ${\rm Gal}(K/k)$-cohomologically trivial.  (See Lemma $5.4.4$ of \cite{Popescu:2002} for instance.)  Hence, by Theorem $5.5.1$ of \cite{Popescu:2002}, Rubin's conjecture is equivalent to Popescu's conjecture.  Computationally, it is more convenient to work with Popescu's conjecture, since one does not have to deal with an auxiliary set of primes $T$ needed in the statement of Rubin's conjecture.  This allows us to focus solely on Popescu's conjecture.  Moreover, in this case, Theorem $4.1.1$ of \cite{Burns:2011} implies that Burns's conjecture follows from the leading term conjecture.

The paper is organized as follows.  We start in \S \ref{preliminaries} with a review of $S$-truncated $L$-functions and the Dirichlet logarithmic map.  In~\S \ref{theory} we gather the necessary theoretical results.
We give a clear definition of an Artin system of $S_{K}$-units in \S \ref{artinsystem} and this allows us to give a description of Stark's regulator in terms of an Artin system of $S_{K}$-units in \S \ref{starkreg}.  We present Stark's main conjecture over $\mathbb{Q}$ in \S \ref{mainst}, Popescu's conjecture in \S \ref{popconj}, and Burns's conjecture in \S \ref{burnsconj}.  We study in detail a very simple example in \S \ref{example} in the order of vanishing one case.  Most of the material contained in \S \ref{theory} is not new, but we rephrase everything in terms of our central notion of an Artin system of $S_{K}$-units.  In the end, \S \ref{artinsystem}, Proposition \ref{concretestarkreg}, Theorem \ref{reformulationQ}, and Proposition \ref{formulaartin} are our main tools that when combined together allow us to provide numerical evidence for Stark's, Rubin's, Popescu's, and Burns's conjectures.  
In~\S \ref{numerical} we explain our numerical calculations.
We outline our method in \S \ref{algorithm} and present the results of our computations with a few examples in \S \ref{data}.
Finally, \S \ref{tables} contains tables that summarize our data.

\subsection*{Acknowledgement}
The authors would like to thank Edward Roualdes and Nicholas Nelson of California State University, Chico
for allowing us to use their computer for our calculations.  

\section{Preliminaries} \label{preliminaries}

\subsection{Basic notation} \label{basic}
Let $k$ be a number field.  We denote its ring of integers by $O(k)$.  A place of $k$ will be denoted by $v$ or $w$.  If $v$ is a finite place then it corresponds to a prime ideal $\mathfrak{p}$ of $O(k)$, and we shall use the words ``place'' or ``prime'' interchangeably.  The corresponding residue field will be denoted by $\kappa(v)$ or $\kappa(\mathfrak{p})$.  Its cardinality is denoted by $\mathbb{N}(v)$ or $\mathbb{N}(\mathfrak{p})$.  To each place $v$, there is an associated normalized absolute value $| \cdot |_{v}$ defined as follows.  Here $\alpha$ denotes an arbitrary element of $k$, and $|\cdot|$ denotes the usual absolute value on $\mathbb{C}$.
\begin{enumerate}
\item If $v$ is a real place with corresponding real embedding $\tau$, then $|\alpha|_{v} = |\tau(\alpha)|$.
\item If $v$ is a complex place with corresponding pair of complex embeddings $\{\tau,\bar{\tau} \}$, then $|\alpha|_{v} = |\tau(\alpha)|^{2}$.
\item If $v$ is a finite place with corresponding prime ideal $\mathfrak{p}$, then $|\alpha|_{v} = \mathbb{N}(\mathfrak{p})^{-{\rm ord}_{\mathfrak{p}}(\alpha)}$, where ${\rm ord}_{\mathfrak{p}}$ is the usual valuation associated to $\mathfrak{p}$.
\end{enumerate}
With these normalizations, we have the product formula:  for all $\alpha \in k^{\times}$,
\begin{equation} \label{prodformula}
\prod_{v}|\alpha|_{v} = 1, 
\end{equation}
where the product is over all places of $k$.  

Throughout this paper, we let $S_{\infty}$ be the set of infinite places of $k$.  The number of real infinite places is denoted by $r_{1}$ and the number of complex infinite places by $r_{2}$.  Hence $|S_{\infty}| = r_{1} + r_{2}$.  Moreover, $S$ will always denote a finite set of places of $k$ that contains $S_{\infty}$.  We have the $S$-integers defined by
$$O_{S}(k) = \{\alpha \in k^{\times} \, | \, {\rm ord}_{v}(\alpha) \ge 0, \text{ for all } v \notin S \}, $$
and we set $E_{S}(k) = O_{S}(k)^{\times}$.  The group $E_{S}(k)$ is known as the group of $S$-units of $k$.  The structure of $E_{S}(k)$ as an abelian group is well-known:  it follows from the $S$-unit theorem that 
$$E_{S}(k) \simeq \mu(k) \times \mathbb{Z}^{|S|-1}, $$
where $\mu(k)$ consists of the roots of unity in $k$.  We set $w_{k} = |\mu(k)|$.

\subsection{The $S$-truncated $L$-functions}  \label{lfunctions}
For simplicity, we shall restrict ourselves to \emph{abelian} extensions of number fields $K/k$.  The Galois group of $K/k$ is denoted by $G$.  As earlier, we fix a finite set of places $S$ of $k$ that is assumed to contain $S_{\infty}$, and we denote the set of places of $K$ lying above places in $S$ by $S_{K}$.  The results of this section are well-known, and we refer the reader to \cite{Lang:1994} for more details.

Given a place $v$ of $k$, one has a short exact sequence
$$1 \longrightarrow I_{v} \longrightarrow G_{v} \longrightarrow {\rm Gal}(\kappa(w)/\kappa(v)) \longrightarrow 1,$$
where $I_{v}$ and $G_{v}$ are the inertia and decomposition group respectively, associated to the place $v$.  We let $\sigma_{v}$ be an element of $G_{v}$ that is mapped to the Frobenius automorphism in ${\rm Gal}(\kappa(w)/\kappa(v))$ via the isomorphism
$G_{v}/I_{v} \stackrel{\simeq}{\longrightarrow} {\rm Gal}(\kappa(w)/\kappa(v)). $
If $v$ is unramified in $K/k$, then $\sigma_{v}$ is unique, since $I_{v} = 1$.  In this case, $\sigma_{v}$ is called the Frobenius automorphism at $v$.

Given a place $v \in S$, we define an element $Fr_{v}$ of $\mathbb{Q}[G]$ as follows:
$$Fr_{v} = \frac{1}{|I_{v}|} \sigma_{v} N_{I_{v}},
\text{ where }
N_{I_{v}} = \sum_{h \in I_{v}}h. $$
Throughout this paper, we denote the trivial character by $\chi_{1}$.  If $\chi \in \widehat{G}$ is such that $\chi \neq \chi_{1}$, then 
\begin{equation}
\chi(Fr_{v}) = 1 \text{ if and only if } G_{v} \subseteq {\rm ker}(\chi).
\end{equation}

Given $\chi \in \widehat{G}$, the corresponding $S$-truncated $L$-function is defined by
$$L_{K,S}(s,\chi) = \prod_{v \notin S}\left(1 - \frac{\chi(Fr_{v})}{\mathbb{N}(v)^{s}} \right)^{-1}. $$
This infinite product converges absolutely and defines a holomorphic function for ${\rm Re}(s) > 1$.

The $L$-functions $L_{K}(s,\chi) := L_{K,S_{\infty}}(s,\chi)$ satisfy a functional equation which we now recall.  Let $\chi \in \widehat{G}$ and let $v$ be a real infinite place.  Then there are two possibilities:  either $G_{v} \subseteq {\rm ker}(\chi)$ or $G_{v} \not\subseteq {\rm ker}(\chi)$.  We let
\begin{enumerate}
\item $r_{1}^{+}(\chi)$ be the number of real infinite places $v$ such that $G_{v} \subseteq {\rm ker}(\chi)$,
\item $r_{1}^{-}(\chi)$ be the number of real infinite places $v$ such that $G_{v} \not\subseteq {\rm ker}(\chi)$.
\end{enumerate}
Define
$$\xi_{k}(s,\chi) = \left(\frac{\sqrt{|\Delta_{k}| \cdot \mathbb{N}(\mathfrak{f(\chi)})}}{2^{r_{2}} \cdot \pi^{d/2}} \right)^{s} \Gamma\left(\frac{1+s}{2} \right)^{r_{1}^{-}(\chi)}\Gamma\left(\frac{s}{2} \right)^{r_{1}^{+}(\chi)} \Gamma(s)^{r_{2}} \cdot L_{K}(s,\chi), $$
where $\Delta_{k}$ is the discriminant of $k$, $\mathfrak{f}(\chi)$ the conductor of the character $\chi$ and $d = [k:\mathbb{Q}]$.  Then
\begin{equation} \label{functeqL}
\xi_{k}(s,\chi) = W(\chi) \cdot \xi_{k}(1-s,\overline{\chi}),
\end{equation}
where $W(\chi)$ is a complex number with absolute value $1$ satisfying $W(\chi_{1}) = 1$.
\begin{The} \label{orderofvanishing}
Let $\chi \in \widehat{G}$ and let $S$ be any finite set of places of $k$ containing $S_{\infty}$.  Then
\begin{equation*}
{\rm ord}_{s=0}L_{K,S}(s,\chi) 
=\begin{cases}
|S| - 1, &\text{ if } \chi = \chi_{1}, \\
|\{v \in S \, | \, G_{v} \subseteq {\rm ker}(\chi)\}|, &\text{ otherwise.} 
\end{cases}
\end{equation*}
\end{The}
\begin{proof}
If $S = S_{\infty}$, then the theorem follows from the known properties of the gamma function, the functional equation (\ref{functeqL}), and the non-trivial fact that $L_{K}(1,\chi) \neq 0$ if $\chi \neq \chi_{1}$.  If $S \neq S_{\infty}$, then the theorem follows by considering what is happening with the Euler factors at the places in $S\smallsetminus S_\infty$.
\end{proof}

The $S_{K}$-truncated Dedekind zeta function of $K$ is defined by
$$\zeta_{K,S}(s) = \prod_{\mathfrak{P} \notin S_{K}}\left(1 - \frac{1}{\mathbb{N}\mathfrak{P}^{s}} \right)^{-1} $$
for ${\rm Re}(s) > 1$ and can be extended to a function that is holomorphic everywhere
except for a simple pole at $s=1$.
Its Taylor expansion at $s=0$ begins as
\begin{equation} \label{classnumberformula}
\zeta_{K,S}(s) = -\frac{h_{K,S}R_{K,S}}{w_{K}}s^{|S_{K}| - 1} + \ldots, 
\end{equation}
where $h_{K,S}$ is the $S_{K}$-class number of $K$ and $R_{K,S}$ the $S_{K}$-regulator.  Note that one can rewrite the order of vanishing of $\zeta_{K,S}$ at $s=0$ as
$${\rm ord}_{s=0}(\zeta_{K,S}) = |S_{K}| - 1 =  {\rm rank}_{\, \mathbb{Z}}E_{S}(K), $$
where we write $E_{S}(K)$ rather than $E_{S_{K}}(K)$ in order to simplify the notation.  The $S_{K}$-truncated Dedekind zeta function can be written in terms of the $S$-truncated $L$-functions as follows:
\begin{equation} \label{prodformulaL}
\zeta_{K,S}(s) = \prod_{\chi \in \widehat{G}}L_{K,S}(s,\chi).
\end{equation}

Let us write
$$L_{K,S}(s,\chi) = c_{S}(\chi)s^{r_{S}(\chi)} + \ldots $$
The order of vanishing $r_{S}(\chi)$ is known due to Theorem \ref{orderofvanishing}.  Combining (\ref{classnumberformula}) and (\ref{prodformulaL}), one has
\begin{equation} \label{prod1}
-\frac{h_{K,S}R_{K,S}}{w_{K}} =  \prod_{\chi \in \widehat{G}}c_{S}(\chi). 
\end{equation}
In the 1970s, Stark proposed a conjectural formula for $c_{S}(\chi)$.  After some preliminaries, we shall present his main conjecture in \S \ref{mainst} below.

If $\chi \in \widehat{G}$, we let
$$e_{\chi} = \frac{1}{|G|}\sum_{\sigma \in G}\chi(\sigma) \cdot \sigma^{-1}, $$
be the corresponding idempotent in the semisimple finite dimensional $\mathbb{C}$-algebra $\mathbb{C}[G]$.
One easily checks that
$$e_{\chi_{1}} = \frac{1}{|G|}N_{G},
\text{ where }
N_{G} = \sum_{\sigma \in G}\sigma. $$
We introduce the $S$-equivariant $L$-function
$$\theta_{K,S}(s) = \sum_{\chi \in \widehat{G}} L_{K,S}(s,\chi) \cdot e_{\overline{\chi}}, $$
which is a meromorphic function from $\mathbb{C}$ into $\mathbb{C}[G]$.  We will also make use of the standard notation $L_{K,S}^{*}(0,\chi)$ instead of $c_{S}(\chi)$ and we set
$$\theta_{K,S}^{*}(0) = \sum_{\chi \in \widehat{G}}L_{K,S}^{*}(0,\chi) \cdot e_{\overline{\chi}}. $$

\subsection{The logarithmic map} \label{logarithm}
We label the places
$$S =\{v_{1},v_{2},\ldots,v_{n} \}, $$
so that $|S| = n$, and in doing so, we introduce an ordering on $S$.  For each $i = 1,\ldots,n$, we fix a place $w_{i}$ of $K$ lying above $v_{i}$.  Following Tate in \cite{Tate:1984}, we let $Y_{S}(K)$ be the free abelian group on the places in $S_{K}$.  We have a short exact sequence of $\mathbb{Z}[G]$-modules
\begin{equation} \label{importantses}
0 \longrightarrow X_{S}(K) \longrightarrow Y_{S}(K) \stackrel{s_{K}}{\longrightarrow} \mathbb{Z} \longrightarrow 0, 
\end{equation}
where the map $s_{K}$ is the augmentation map and $X_{S}(K)$ its kernel.  Recall that $s_K$ is defined by setting $s_{K}(w) = 1$ for all $w \in S_{K}$ and extending by linearity.

If $A$ if a finite abelian group, $M$ a $\mathbb{Z}[A]$-module and $F$ a subfield of $\mathbb{C}$, then we write $FM$ rather than $F\otimes_{\mathbb{Z}}M$.  We define the logarithmic map  
$$\lambda_{K,S}: E_{S}(K) \longrightarrow \mathbb{C} Y_{S}(K)$$
by the formula
\begin{equation*}
\lambda_{K,S}(u) = -\sum_{w \in S_{K}} \log|u|_{w} \cdot  w,
\end{equation*}
whenever $u \in E_{S}(K)$.  Because of the product formula (\ref{prodformula}), $\lambda_{K,S}$ takes values in $\mathbb{C} X_{S}(K)$.  Its extension to $\mathbb{C} E_{S}(K)$ will be denoted by the same symbol.  Not only is this map a $\mathbb{C}$-linear map, but it is also $G$-equivariant; hence, it is a $\mathbb{C}[G]$-module morphism.  The $S_{K}$-unit theorem implies that $\lambda_{K,S}$ induces an isomorphism of $\mathbb{C}[G]$-modules
\begin{equation} \label{Gisom}
\lambda_{K,S}: \mathbb{C} E_{S}(K) \stackrel{\simeq}{\longrightarrow} \mathbb{C} X_{S}(K). 
\end{equation}

Recall that we have an injection of $\mathbb{Z}[G]$-modules $\iota_{k}:Y_{S}(k) \hookrightarrow Y_{S}(K)$ defined by
$$v \mapsto |G_{v}| \sum_{w \, | \, v }w. $$
(The group $G$ acts trivially on $Y_{S}(k)$.)
After tensoring with $\mathbb{C}$, we get an injective morphism of $\mathbb{C}[G]$-modules $\mathbb{C}Y_{S}(k) \hookrightarrow \mathbb{C}Y_{S}(K)$ that we denote by the same symbol $\iota_{k}$.  This map allows us to view $\mathbb{C}Y_{S}(k)$ inside of $\mathbb{C}Y_{S}(K)$, so that $\mathbb{C}Y_{S}(k) \subseteq \mathbb{C}Y_{S}(K)$.

\begin{Pro} \label{clear}
Let $\chi \in \widehat{G}$.
\begin{enumerate}
\item If $\chi \neq \chi_{1}$, then \label{unun}
$$\mathbb{C}Y_{S}(K) \cdot e_{\chi} = \mathbb{C}X_{S}(K) \cdot e_{\chi}. $$
\item For the trivial character, we have \label{deuxdeux}
$$\mathbb{C}Y_{S}(K) \cdot e_{\chi_{1}} = \mathbb{C}Y_{S}(k) \text{ and } \mathbb{C}X_{S}(K) \cdot e_{\chi_{1}} = \mathbb{C}X_{S}(k). $$
\end{enumerate}
\end{Pro}
\begin{proof}
Tensoring the short exact sequence (\ref{importantses}) with $\mathbb{C}$ gives the short exact sequence of $\mathbb{C}[G]$-modules
$$0 \longrightarrow \mathbb{C}X_{S}(K) \longrightarrow \mathbb{C}Y_{S}(K) \longrightarrow \mathbb{C} \longrightarrow 0.$$
Since $G$ acts trivially on $\mathbb{C}$, part (\ref{unun}) follows at once.
To show that $\mathbb{C}Y_{S}(K) \cdot e_{\chi_{1}} = \mathbb{C}Y_{S}(k)$, the main point is to 
use the equality
\begin{equation*}
e_{\chi_{1}} \cdot w = \frac{1}{|G|}\iota_{k}(v)
\end{equation*}
that is valid for all places $v \in S$ and all places $w$ of $K$ lying above $v$.
It then immediately follows that $\mathbb{C}X_{S}(K) \cdot e_{\chi_{1}} = \mathbb{C}X_{S}(k)$.
\end{proof}

The following proposition, though simple, is quite useful.
\begin{Pro} \label{useful}
Let $\chi \in \widehat{G}$ be a non-trivial character.  Furthermore, let $v \in S$ and let $w$ be a place of $K$ lying above $v$.  In $\mathbb{C}Y_{S}(K)$, we have
\begin{enumerate}
\item If $G_{v} \subseteq {\rm ker}(\chi)$, then $e_{\chi} \cdot w \neq 0$, 
\item If $G_{v} \not \subseteq {\rm ker}(\chi)$, then $e_{\chi} \cdot w = 0$. 
\end{enumerate}
\end{Pro}
\begin{proof}
Let $\sigma_{1},\ldots,\sigma_{s}$ be a complete set of representatives of $G/G_{v}$.  Then
\begin{equation*}
\begin{aligned}
e_{\chi} \cdot w &= \frac{1}{|G|}\sum_{\sigma \in G}\chi(\sigma)\sigma^{-1} \cdot w \\
&= \frac{1}{|G|}\sum_{i=1}^{s}\sum_{h \in G_{v}}\chi(\sigma_{i}h)\sigma_{i}^{-1}h^{-1} \cdot w\\
&= \frac{1}{|G|} \sum_{i=1}^{s} \chi(\sigma_{i})\sigma_{i}^{-1} \cdot w \cdot \left(\sum_{h \in G_{v}}\chi(h) \right).
\end{aligned}
\end{equation*}
If $G_{v} \subseteq {\rm ker}(\chi)$, then this last line is 
$$\frac{|G_{v}|}{|G|}\sum_{i=1}^{s}\chi(\sigma_{i}) \sigma_{i}^{-1}\cdot w \neq 0.$$
On the other hand, if $G_{v} \not \subseteq {\rm ker}(\chi)$, then we get zero, since
$$\sum_{h \in G_{v}} \chi(h) = 0.$$
\end{proof}

Using the previous proposition,
one can give a different formula for the order of vanishing of the $S$-truncated $L$-functions.
\begin{The} \label{orderofvanishing1}
Let $\chi \in \widehat{G}$.  Then
\begin{equation*}
\begin{aligned}
{\rm ord}_{s=0}L_{K,S}(s,\chi) = {\rm dim}_{\mathbb{C}}(\mathbb{C}X_{S}(K) \cdot e_{\chi}) = {\rm dim}_{\mathbb{C}}(\mathbb{C}E_{S}(K) \cdot e_{\chi}).
\end{aligned} 
\end{equation*}
\end{The}
\begin{proof}
The first equality follows from Theorem \ref{orderofvanishing}, Proposition \ref{useful} and Proposition \ref{clear}.  The second one follows from the isomorphism (\ref{Gisom}).
\end{proof}

For each $i=1,\ldots,n$, we define $\ell_{i,K}: E_{S}(K) \longrightarrow \mathbb{C}[G]$ by the formula
$$\ell_{i,K}(u) = - \frac{1}{|G_{i}|}\sum_{\sigma \in G}\log|u^{\sigma}|_{w_{i}}\cdot \sigma^{-1}, $$
where from now on we write $G_{i}$ rather than $G_{v_{i}}$.  Its extension to $\mathbb{C}E_{S}(K)$ will also be denoted by $\ell_{i,K}$.  Note that the maps $\ell_{i,K}$ are $G$-equivariant.  
\begin{Pro} \label{logdec}
For $x \in \mathbb{C}E_{S}(K)$, we have
\begin{equation*}
\begin{aligned}
\lambda_{K,S}(x) = \sum_{i=1}^{n}\ell_{i,K}(x) \cdot w_{i}
\end{aligned}
\end{equation*}
\end{Pro}
\begin{proof}
Let $\sigma_{1},\ldots,\sigma_{s}$ be a complete set of representatives of $G/G_{i}$ and let $u \in E_{S}(K)$.  Then
\begin{equation*}
\begin{aligned}
\sum_{i=1}^{n}\ell_{i,K}(u) \cdot w_{i} &= \sum_{i=1}^{n}\left(-\frac{1}{|G_{i}|}\sum_{\sigma \in G}\log|u^{\sigma}|_{w_{i}} \cdot \sigma^{-1} \right) w_{i}\\
&= - \sum_{i=1}^{n} \frac{1}{|G_{i}|}\sum_{t = 1}^{s}\sum_{h \in G_{i}} \log|u^{\sigma_{t}h}|_{w_{i}}(\sigma_{t}h)^{-1} \cdot w_{i}\\
&= - \sum_{i=1}^{n}\sum_{t=1}^{s} \log|u^{\sigma_{t}}|_{w_{i}} \sigma_{t}^{-1} \cdot w_{i} \\
&= \lambda_{K,S}(u)
\end{aligned}
\end{equation*}
\end{proof}
Let us look at the behavior of the maps $\ell_{i,K}$ on various isotypical components of $\mathbb{C}E_{S}(K)$.
For the next proposition, it may be helpful to observe
that $\ell_{i,k}:\mathbb{C}E_{S}(k) \longrightarrow \mathbb{C}$ is the $\mathbb{C}[G]$-module morphism defined by
$\ell_{i,k}(u) = - \log|u|_{v_{i}}$
for $u \in E_{S}(k)$.  ($G$ acts trivially here.)
\begin{Pro} \label{isotypical}
Let $\chi \in \widehat{G}$ and $i\in\{1,\dots,n\}$.
\begin{enumerate}
\item
Suppose $\chi\neq\chi_{1}$ and
$G_{i} \not \subseteq {\rm ker}(\chi)$.  If $x \in \mathbb{C}E_{S}(K) \cdot e_{\chi}$, then
$$\ell_{i,K}(x) = 0. $$ \label{u}
\item
If $x \in \mathbb{C}E_{S}(K) \cdot e_{\chi_{1}}$, then
$$\ell_{i,K}(x) = \ell_{i,k}(x) \cdot N_{G}.$$ \label{d}
\end{enumerate}
\end{Pro}
\begin{proof}
For (\ref{u}), we proceed as follows.  Note that if $h \in G_{i}$, then $\ell_{i,K}(x) = h \cdot \ell_{i,K}(x)$ for all $x \in \mathbb{C}E_{S}(K)$.  Hence, for all $h \in G_{i}$, we have
$$e_{\chi} \cdot \ell_{i,K}(x) = \chi(h) e_{\chi}\cdot \ell_{i,K}(x). $$  
Summing over all $h \in G_{i}$ gives
$$|G_{i}| \cdot e_{\chi} \cdot \ell_{i,K}(x) = \left( \sum_{h \in G_{i}} \chi(h) \right) e_{\chi} \cdot \ell_{i,K}(x). $$
Since $G_{i} \not\subseteq {\rm ker}(\chi)$, we have
$$\sum_{h \in G_{i}}\chi(h) = 0, $$
and this proves (\ref{u}).  Point (\ref{d}) is a simple calculation left to the reader.
\end{proof}

If $r$ is an integer satisfying $1\le r \le |S|$, we want to study the $\mathbb{C}[G]$-module morphism
\begin{equation} \label{wedgelog}
\wedge^{r}\lambda_{K,S}: \bigwedge_{\mathbb{C}[G]}^{r}\mathbb{C}E_{S}(K) \longrightarrow \bigwedge_{\mathbb{C}[G]}^{r}\mathbb{C}Y_{S}(K), 
\end{equation}
defined on pure wedges by the formula $\wedge^{r}\lambda_{K,S}(x_{1}\wedge \ldots \wedge x_{r}) = \lambda_{K,S}(x_{1})\wedge \ldots \wedge \lambda_{K,S}(x_{r})$.
In order to do so, we let
$$\Omega = \{1,\ldots,n \}. $$
For any totally ordered set $\mathcal{X}$, such as $\Omega$, the symbol $\wp_{r}(\mathcal{X})$ will denote the set of $r$-tuples $(x_{1},\ldots,x_{r})$, where $x_{i} \in \mathcal{X}$, and $x_{1}<\ldots<x_{r}$.  If $I \in \wp_{r}(\Omega)$ is such that $I = (i_{1},\ldots,i_{r})$ with $1\le i_{1}< \ldots<i_{r} \le n$, then we set
$$w_{I} = w_{i_{1}}\wedge\ldots\wedge w_{i_{r}} \in \bigwedge_{\mathbb{C}[G]}^{r}\mathbb{C}Y_{S}(K), $$
and we define $R_{I,K}:\bigwedge_{\mathbb{C}[G]}^{r}\mathbb{C}E_{S}(K) \longrightarrow \mathbb{C}[G]$ by the following formula on pure wedges
$$R_{I,K}(x_{1}\wedge \ldots \wedge x_{r}) = {\rm det}\left(\ell_{i_{s},K}(x_{t}) \right)_{s,t=1,\ldots,r}. $$
Note that for all $I \in \wp_{r}(\Omega)$,
the map $R_{I,K}$ is a $\mathbb{C}[G]$-module morphism.
It is worth pointing out that the morphism (\ref{wedgelog}) is injective, since $\mathbb{C}[G]$ is semisimple.  Combining with Proposition \ref{logdec}, one gets for a pure wedge $x = x_{1}\wedge \ldots \wedge x_{r} \in \bigwedge^{r}_{\mathbb{C}[G]} \mathbb{C}E_{S}(K)$ the formula
\begin{equation}
\wedge^{r}\lambda_{K,S}(x) = \sum_{I \in \wp_{r}(\Omega)}R_{I,K}(x) \cdot w_{I}
\end{equation}
Finally, we see what happens when we restrict the maps $R_{I,K}$ to some isotypical components of $\bigwedge_{\mathbb{C}[G]}^{r}\mathbb{C}E_{S}(K)$.
\begin{Pro} \label{isotyp}
Let $\chi \in \widehat{G}$ be such that $\chi \neq \chi_{1}$.  Moreover, let $I \in \wp_{r}(\Omega)$ be such that there exists $i \in I$ for which $G_{i} \not \subseteq {\rm ker}(\chi)$.  Then, 
$$R_{I,K}(x) = 0$$
for all $x \in \bigwedge_{\mathbb{C}[G]}^{r}\mathbb{C}E_{S}(K) \cdot e_{\chi}$.
\end{Pro}
\begin{proof}
If $x_{1},\ldots,x_{r} \in \mathbb{C}E_{S}(K)$, then
\begin{equation*}
  \begin{aligned}
    R_{I,K}(x_{1}\wedge \ldots \wedge x_{r} \cdot e_{\chi}) &= R_{I,K}((x_{1} \cdot e_{\chi}) \wedge \ldots \wedge (x_{r} \cdot e_{\chi})) \\
    &= {\rm det}(\ell_{i_{s},K}(x_{t}\cdot e_{\chi}))_{s,t=1,\ldots,r}
  \end{aligned}
\end{equation*}
Now, there exists $s \in \{1,\ldots,r \}$ such that $G_{i_{s}} \not \subseteq {\rm ker}(\chi)$.  Hence, Proposition \ref{isotypical} implies 
$$\ell_{i_{s},K}(x_{t} \cdot e_{\chi}) = 0 $$
for all $t=1,\ldots,r$.
\end{proof}

\section{Stark's conjecture} \label{theory}

\subsection{Artin systems of $S_{K}$-units} \label{artinsystem}
Recall that Stark's original idea was to break down the $S_{K}$-regulator into $\chi$ components, and Artin systems of $S_{K}$-units played an important role in doing so.  In this section, we give a definition for an Artin system of $S_{K}$-units.  As before, $K/k$ is an abelian extension of number fields, and $S$ is a finite set of places of $k$ containing $S_{\infty}$.  We start with the following definition.
\begin{Def}
An Artin system of $S_{K}$-units $\mathcal{A}$ is a collection of $S_{K}$-units
$$\mathcal{A} = \{\varepsilon_{w} \, | \, w \in S_{K} \} \subseteq E_{S}(K), $$
such that the group morphism 
$$f:Y_{S}(K) \longrightarrow E_{S}(K)$$ 
defined by $w \mapsto \varepsilon_{w}$ satisfies the following properties:
\begin{enumerate}
\item $f$ is $G$-equivariant,
\item ${\rm ker}(f) = \mathbb{Z} \cdot \alpha$ for some $\alpha \in Y_{S}(K)^{G}$ that satisfies $s_{K}(\alpha) \neq 0$.
\end{enumerate}
\end{Def}
Note that $G$ acts trivially on $\mathbb{Z} \cdot \alpha$.  Moreover, since
$${\rm rank}_{\mathbb{Z}}\left(Y_{S}(K)/\mathbb{Z} \cdot \alpha\right) = |S_{K}|-1, $$
one has ${\rm coker}(f)$ is finite.  As a result, an Artin system of $S_{K}$-units can be conveniently described by an exact sequence of $\mathbb{Z}[G]$-modules
\begin{equation} \label{artinles}
  0 \longrightarrow \mathbb{Z} \cdot \alpha \longrightarrow Y_{S}(K) \stackrel{f}\longrightarrow E_{S}(K) \longrightarrow A \longrightarrow 0,
\end{equation}
where $A$ is some $\mathbb{Z}[G]$-module with finite cardinality.  Letting $d_{0} = s_{K}(\alpha)$, and applying the snake lemma to the commutative diagram
\begin{equation*}
        \begin{CD}
                0 @>>> 0 @>>> \mathbb{Z} \cdot \alpha @>{s_{K}}>> \mathbb{Z} \cdot d_{0} @>>>0 \\
                & & @VVV @VVV @VVV \\
                0  @>>> X_{S}(K) @>>> Y_{S}(K) @>{s_{K}}>> \mathbb{Z} @>>> 0\\
        \end{CD}
\end{equation*}
leads to the short exact sequence of $\mathbb{Z}[G]$-modules
$$0 \longrightarrow X_{S}(K) \longrightarrow Y_{S}(K)/\mathbb{Z} \cdot \alpha \longrightarrow \mathbb{Z}/d_{0} \mathbb{Z} \longrightarrow 0.$$
Therefore, the morphism $f:Y_{S}(K) \longrightarrow E_{S}(K)$ induces by restriction an injective morphism of $\mathbb{Z}[G]$-modules
\begin{equation} \label{injectiveX}
f: X_{S}(K) \hookrightarrow E_{S}(K). 
\end{equation}
In particular, the image of $X_{S}(K)$ via $f$ gives a group of $S_{K}$-units that is of finite index in $E_{S}(K)$.
\begin{Pro} \label{unitschi}
Let $\mathcal{A} =\{\varepsilon_{w}\, | \, w \in S_{K} \}$ be an Artin system of $S_{K}$-units, and let $\chi \in \widehat{G}$ be a non-trivial character.  Furthermore, let $v \in S$ and let $w$ be a place of $K$ lying above $v$.  In $\mathbb{C}E_{S}(K)$, we have
\begin{enumerate}
\item If $G_{v} \subseteq {\rm ker}(\chi)$, then $\varepsilon_{w}\cdot e_{\chi} \neq 0$,
\item If $G_{v} \not \subseteq {\rm ker}(\chi)$, then $\varepsilon_{w} \cdot e_{\chi} = 0$.
\end{enumerate}
\end{Pro}
\begin{proof}
Tensoring the exact sequence (\ref{artinles}) with $\mathbb{C}$ leads to the short exact sequence of $\mathbb{C}[G]$-modules:
$$0 \longrightarrow \mathbb{C} \cdot \alpha \longrightarrow \mathbb{C}Y_{S}(K) \stackrel{f_{\mathbb{C}}}{\longrightarrow} \mathbb{C}E_{S}(K) \longrightarrow 0. $$
Since $\chi \neq \chi_{1}$ and $G$ acts trivially on $\mathbb{C} \cdot \alpha$, we get an isomorphism of $\mathbb{C}[G]$-modules
$$f_{\mathbb{C}}^{\chi}:\mathbb{C}Y_{S}(K) \cdot e_{\chi} \stackrel{\simeq}{\longrightarrow} \mathbb{C}E_{S}(K)\cdot e_{\chi}. $$
The result then follows from Proposition \ref{useful}.
\end{proof}

From now on, for $v \in S$, we let 
$$T_{v}(K) = \sum_{w \, | \, v}w \in Y_{S}(K).$$
Note that $\alpha \in Y_{S}(K)$ is fixed by $G$ if and only if
$$\alpha = \sum_{v \in S}n_{v} \cdot T_{v}(K), $$
for some $n_{v} \in \mathbb{Z}$.

\begin{The} \label{artinexist}
Let $K/k$ be a finite abelian extension of number fields and let $S$ be a finite set of places of $k$ containing $S_{\infty}$.  Then there exist Artin systems of $S_{K}$-units.
\end{The}
\begin{proof}
For the reader's convenience, we include here the proof contained in \cite{Artin:1932}.  (In \cite{Artin:1932}, only the case $S = S_{\infty}$ was treated, but the argument works for any finite set of places $S$ that contains $S_{\infty}$.)  For each $i=1,\ldots,n$, let $\beta_{i} \in E_{S}(K)$ be such that
\begin{enumerate}
\item $|\beta_{i}|_{w_{i}} > 1$
\item $|\beta_{i}|_{w}< 1$ for all $w \in S_{K}$ satisfying $w \neq w_{i}$.
\end{enumerate}
The existence of $S_{K}$-units with those properties is a well-known result of algebraic number theory.
See \S~$1$ of Chapter~$\textrm{V}$ in \cite{Lang:1994} for instance.  Then set
$$\gamma_{i} = \beta_{i}^{N_{i}}, $$
where $N_{i} = N_{G_{i}}$.  Note that $\gamma_{i} \in K^{G_{i}}$.  A simple calculation shows that the $S_{K}$-units $\gamma_{i}$ still satisfy
\begin{enumerate}
\item $|\gamma_{i}|_{w_{i}} > 1$
\item $|\gamma_{i}|_{w} < 1$ for all $w \in S_{K}$ satisfying $w \neq w_{i}$.
\end{enumerate}
If $w \, | \, v_{i}$, then there exists $\sigma \in G$ such that $w = w_{i}^{\sigma}$.  Set
$$\varepsilon_{w} = \gamma_{i}^{\sigma}. $$
The $S_{K}$-units $\varepsilon_{w}$ do not depend on the choice of $\sigma$.  Moreover, they satisfy
$$\varepsilon_{w}^{\tau} = \varepsilon_{w^{\tau}}$$
for all $\tau \in G$ and also
\begin{enumerate}
\item $|\varepsilon_{w}|_{w} > 1$
\item $|\varepsilon_{w}|_{w'} < 1$ for all places $w' \in S_{K}$ satisfying $w' \neq w$.
\end{enumerate}
By Lemma \ref{specialdet} below, removing any $S_{K}$-unit from the set $\{\varepsilon_{w} \, | \, w \in S_{K}\}$ gives a system of independent $S_{K}$-units.  Therefore, there is precisely one relation among them, say
$$\prod_{w \in S_{K}}\varepsilon_{w}^{n_{w}} = 1 $$
for some integers $n_{w}$.  Since the group $G$ acts transitively on the places lying above a fixed place $v \in S$, we see that for all $v \in S$, there exists $n_{v} \in \mathbb{Z}$ such that $n_{w} = n_{v}$ whenever $v \, | \, w$.  Taking the inverses of some of the $\varepsilon_{w}$ if necessary, one can assume that $n_{v} \ge 0$ for all $v \in S$.  The set $\{\varepsilon_{w}\, | \, w \in S_{K}  \}$ is the desired Artin system of $S_{K}$-units, since the kernel of the $\mathbb{Z}[G]$-module morphism $f:Y_{S}(K) \longrightarrow E_{S}(K)$ defined by $f(w) = \varepsilon_{w}$ is $\mathbb{Z} \cdot \alpha$, where
$$\alpha = \sum_{v \in S}n_{v} \cdot T_{v}(K), $$
and $s_{K}(\alpha) = \sum_{v \in S}n_{v} \frac{|G|}{|G_{v}|} \neq 0$.
\end{proof}
\begin{Rem}
It is always possible to take 
$\alpha = \sum_{v \in S}T_{v}(K)$
just by setting $\delta_{w} = \varepsilon_{w}^{n_{v}}$ in the last proof.  Then one has
$$\prod_{w \in S_{K}}\delta_{w} = 1. $$
Numerically, it is more convenient to allow any $\alpha \in Y_{S}(K)^{G}$, because the index
$$m = [E_{S}(K):\mu(K) \cdot f(X_{S}(K))] $$
is usually smaller.
\end{Rem}

The following lemma is simple and we skip its proof.
\begin{Lem} \label{specialdet}
Let $A = (a_{ij}) \in M_{n}(\mathbb{R})$ be a matrix satisfying
\begin{enumerate}
\item $a_{ij}<0$ whenever $i \neq j$,
\item $\sum_{j=1}^{n}a_{ij}>0$ for all $i=1,\ldots,n$.
\end{enumerate}
Then ${\rm det}(A) \neq 0$.
\end{Lem}

\begin{Rem}
We remark that an Artin system of $S_{K}$-units exists as well in the case of a non-abelian Galois extension $K/k$, but we restrict ourselves to the abelian case in this paper.
\end{Rem}

\subsection{The Stark regulator} \label{starkreg}
If $L$ is any number field, let us start by reminding the reader about the regulator of a subgroup of units of $L$.  For the moment, we fix a finite set of places $S$ of $L$, and we let $n = |S|$.
\begin{Def}
Given a subgroup $\mathcal{U}$ of $E_{S}(L)$ such that $E_{S}(L)/\mathcal{U}$ is finite, we define ${\rm Reg}_{L,S}(\mathcal{U}) \in \mathbb{R}/\{\pm 1 \}$ as follows.  If $\{\eta_{1},\ldots,\eta_{n-1}\}$ is a set of units whose classes in $\mathcal{U}/\mathcal{U}_{{\rm tor}}$ form a $\mathbb{Z}$-basis, then consider the matrix 
\begin{equation} \label{mat}
\left(\log|\eta_{j}|_{w} \right) \in M_{n,n-1}(\mathbb{R}), 
\end{equation}
where $j=1,\ldots, n-1$ and $w \in S$.  The regulator ${\rm Reg}_{L,S}(\mathcal{U})$ is defined to be the determinant of the matrix (\ref{mat}) after removing one row.
\end{Def}
Note that removing a different row or choosing another $\mathbb{Z}$-basis for $\mathcal{U}/\mathcal{U}_{{\rm tor}}$ will change the determinant by at most a sign.  Hence, this definition makes sense modulo $\{\pm 1\}$.  Also, we have
$$R_{L,S} = |{\rm Reg}_{L,S}(E_{S}(L))|. $$
The following proposition is well-known, and we skip its proof.
\begin{Pro} \label{subreg}
Given a subgroup $\mathcal{U}$ of $E_{S}(L)$ such that $E_{S}(L)/\mathcal{U}$ is finite, we have
$$[E_{S}(L):\mu(L) \cdot \mathcal{U}] = \frac{|{\rm Reg}_{L,S}(\mathcal{U})|}{R_{L,S}} $$
\end{Pro}
We now go back to our setting where $K/k$ is a finite abelian extension of number fields and $S$ is a finite set of places of $k$ containing $S_{\infty}$.  Even though it is not clear how to break up $R_{K,S}$ into $\chi$-components, it is possible to do so with ${\rm Reg}_{K,S}(\mathcal{U})$, where $\mathcal{U}$ is a group of $S_{K}$-units coming from an Artin system of $S_{K}$-units.  (Here we write ${\rm Reg}_{K,S}(\mathcal{U})$ rather than ${\rm Reg}_{K,S_{K}}(\mathcal{U})$ in order to simplify the notation.)  This fact was recognized by Stark in \cite{Stark:1975}, and this gives a way of breaking up $R_{K,S}$ into $\chi$-components, at least up to a rational number, namely the index $[E_{S}(K):\mu(K) \cdot \mathcal{U}]$.  Here is how this works.  Starting with an Artin system of $S_{K}$-units $\mathcal{A}$ and its corresponding morphism $f: Y_{S}(K) \longrightarrow E_{S}(K)$, we have an induced morphism of $\mathbb{C}[G]$-modules $f_{\mathbb{C}}:\mathbb{C}Y_{S}(K) \longrightarrow \mathbb{C}E_{S}(K)$.  We will now look at the isomorphism of $\mathbb{C}[G]$-modules
$$f_{\mathbb{C}} \circ \lambda_{K,S}:\mathbb{C}E_{S}(K) \longrightarrow \mathbb{C}E_{S}(K). $$
Since this map is a linear endomorphism of the $\mathbb{C}$-vector space $\mathbb{C}E_{S}(K)$, we can talk about its determinant.  Recall also that from (\ref{injectiveX}), $f(X_{S}(K))$ is a group of finite index in $E_{S}(K)$.
\begin{Pro} \label{reg}
Let $\mathcal{A} = \{\varepsilon_{w} \, | \, w \in S_{K} \}$ be an Artin system of $S_{K}$-units with corresponding morphism $f$.  Then
$${\rm det}(f_{\mathbb{C}} \circ \lambda_{K,S}) = \pm {\rm Reg}_{K,S}(\mathcal{U}_{f}), $$
where $\mathcal{U}_{f} = f(X_{S}(K))$.
\end{Pro}
\begin{proof}
Let $w_{0}$ be any place of $S_{K}$.  Note that the images of the $S_{K}$-units $\{\varepsilon_{w}\varepsilon_{w_{0}}^{-1} \, | \, w \in S_{K}, w \neq w_{0} \}$ in $\mathbb{C}E_{S}(K)$ form a basis of the $\mathbb{C}$-vector space $\mathbb{C}E_{S}(K)$.  These $S_{K}$-units also form a $\mathbb{Z}$-basis of $\mathcal{U}_{f}$ modulo its torsion subgroup.  We calculate
\begin{equation*}
\begin{aligned}
f_{\mathbb{C}} \circ \lambda_{K,S}(\varepsilon_{w}\varepsilon_{w_{0}}^{-1}) &= f_{\mathbb{C}}\left(\sum_{v \in S_{K}}\log|\varepsilon_{w}\varepsilon_{w_{0}}^{-1}|_{v} \cdot v \right) \\
&= \sum_{\substack{v \in S_{K} \\ v \neq w_{0}}}\log|\varepsilon_{w} \varepsilon_{w_{0}}^{-1}|_{v} \cdot (\varepsilon_{v}\varepsilon_{w_{0}}^{-1}).
\end{aligned}
\end{equation*}
Hence ${\rm det}(f_{\mathbb{C}} \circ \lambda_{K,S}) = \pm {\rm Reg}_{K,S}(\mathcal{U}_{f})$ as we wanted to show.
\end{proof}
Moreover, $f_{\mathbb{C}} \circ \lambda_{K,S}$ is a morphism of $\mathbb{C}[G]$-modules, we have
\begin{equation} \label{prodreg}
  {\rm det}(f_{\mathbb{C}} \circ \lambda_{K,S}) = \prod_{\chi \in \widehat{G}}{\rm det}\left((f_{\mathbb{C}} \circ \lambda_{K,S})^{\chi} \right).
\end{equation}

\begin{Def}
  Let $\mathcal{A}$ be an Artin system of $S_{K}$-units with corresponding morphism $f$.  Given $\chi \in \widehat{G}$, one defines the Stark regulator associated to $\chi$ and $\mathcal{A}$ to be
  $$R(\chi,\mathcal{A}) = {\rm det}\left( (f_{\mathbb{C}} \circ \lambda_{K,S})^{\overline{\chi}} \right). $$
\end{Def}

Combining Proposition \ref{subreg}, Proposition \ref{reg} and (\ref{prodreg}) leads to the formula
\begin{equation} \label{regulatorbroken}
  \pm R_{K,S} = \frac{1}{[E_{S}(K):\mu(K)\cdot \mathcal{U}_{f}]}\prod_{\chi \in \widehat{G}} R(\chi,\mathcal{A}).
\end{equation}

We now present an alternative description of the Stark regulator.  Given an Artin system of $S_{K}$-units $\mathcal{A} = \{\varepsilon_{w} \, | \, w \in S_{K} \}$ and an integer $i$ satisfying $1 \le i \le |S|$, we will write $\varepsilon_{i}$ rather than $\varepsilon_{w_{i}}$.
\begin{Pro} \label{concretestarkreg}
Let $\mathcal{A} = \{\varepsilon_{w}\, | \, w \in S_{K} \}$ be an Artin system of $S_{K}$-units.
\begin{enumerate} 
\item Let $\chi \in \widehat{G}$ be such that $\chi \neq \chi_{1}$.  Let \label{unnn}
$$r = {\rm ord}_{s=0}L_{K,S}(s,\chi), $$
and let $I = (i_{1},\ldots,i_{r})$ be the unique element of $\wp_{r}(\Omega)$ such that $G_{i_{t}} \subseteq {\rm ker}(\chi)$ for all $t=1,\ldots,r$.  Then, one has 
$$R(\chi,\mathcal{A}) = \overline{\chi}\left(R_{I,K}(\varepsilon_{i_{1}} \wedge \ldots \wedge \varepsilon_{i_{r}}) \right).$$
\item Let $\chi = \chi_{1}$ be the trivial character.  Let  \label{deuxxx}
$$r = {\rm ord}_{s=0}L_{K,S}(s,\chi_{1}) = |S| - 1 $$
and let $I = (i_{1},\ldots,i_{r})$ be any element of $\wp_{r}(\Omega)$.  Then, one has
$$R(\chi,\mathcal{A}) = \chi_{1}\left(R_{I,K}\left((\varepsilon_{i_{1}}\varepsilon_{i_{r+1}}^{-1})\wedge \ldots \wedge(\varepsilon_{i_{r}}\varepsilon_{i_{r+1}}^{-1}) \right) \right), $$
where $i_{r+1}$ is the unique index in $\Omega$ that is not in $I$.
\end{enumerate}
\end{Pro}
\begin{proof}
Starting with the exact sequence (\ref{artinles}), one gets the following short exact sequence of $\mathbb{C}[G]$-modules:
$$0 \longrightarrow \mathbb{C} \cdot \alpha \longrightarrow \mathbb{C}Y_{S}(K) \stackrel{f_{\mathbb{C}}}{\longrightarrow} \mathbb{C}E_{S}(K) \longrightarrow 0.$$
Since $G$ acts trivially on $\mathbb{C} \cdot \alpha$, and $\chi \neq \chi_{1}$, one gets an isomorphism of $\mathbb{C}[G]$-modules:
$$f_{\mathbb{C}}^{\overline{\chi}}:\mathbb{C}Y_{S}(K)\cdot e_{\overline{\chi}} \longrightarrow \mathbb{C}E_{S}(K) \cdot e_{\overline{\chi}}. $$
By Proposition \ref{useful}, a $\mathbb{C}$-basis for $\mathbb{C}Y_{S}(K) \cdot e_{\overline{\chi}}$ is given by $\{w_{i_{t}} \cdot e_{\overline{\chi}}\, | \, t = 1, \ldots, r \}$ and therefore, a $\mathbb{C}$-basis for $\mathbb{C}E_{S}(K) \cdot e_{\overline{\chi}}$ is given by $\{\varepsilon_{i_{t}} \cdot e_{\overline{\chi}} \, | \, t=1,\ldots,r \}$.  It follows that a $\mathbb{C}$-basis for the one-dimensional $\mathbb{C}$-vector space $\bigwedge_{\mathbb{C}[G]}^{r}\mathbb{C}E_{S}(K) \cdot e_{\overline{\chi}}$ is given by $\varepsilon_{i_{1}}\wedge \ldots \wedge \varepsilon_{i_{r}} \cdot e_{\overline{\chi}}$.  Using Proposition \ref{isotyp}, we calculate
\begin{equation*}
\begin{aligned}
\wedge^{r}(f_{\mathbb{C}} \circ \lambda_{K,S})(\varepsilon_{i_{1}} \wedge \ldots \wedge \varepsilon_{i_{t}} \cdot e_{\overline{\chi}}) &= \wedge^{r}f_{\mathbb{C}} \left(\sum_{J \in \wp_{r}(\Omega)}R_{J,K}(\varepsilon_{i_{1}}\wedge \ldots \wedge \varepsilon_{i_{r}} \cdot e_{\overline{\chi}}) \cdot w_{J} \right) \\
&=\wedge^{r}f_{\mathbb{C}} \left(R_{I,K}(\varepsilon_{i_{1}}\wedge \ldots \wedge \varepsilon_{i_{r}} \cdot e_{\overline{\chi}}) \cdot w_{I} \right) \\
&= \overline{\chi}(R_{I,K}(\varepsilon_{i_{1}}\wedge \ldots \wedge \varepsilon_{i_{r}})) \cdot \varepsilon_{i_{1}}\wedge \ldots \wedge \varepsilon_{i_{r}} \cdot e_{\overline{\chi}},
\end{aligned}
\end{equation*}
and this shows (\ref{unnn}).  

For (\ref{deuxxx}), we proceed as follows.  Let $I \in \wp_{r}(\Omega)$ and let $\Omega \smallsetminus I = \{i_{r+1}\}$.  Since $\{w_{i} \cdot e_{\chi_{1}} \, | \, i=1,\ldots,|S| \}$ is a $\mathbb{C}$-basis for $\mathbb{C}Y_{S}(K) \cdot e_{\chi_{1}}$, we get that
$$\{(w_{i_{s}} - w_{i_{r+1}}) \cdot e_{\chi_{1}}\, | \, s=1,\ldots,r\} $$
is a $\mathbb{C}$-basis for $\mathbb{C}X_{S}(K) \cdot e_{\chi_{1}}$.  The isomorphism
$$f:\mathbb{C}X_{S}(K) \cdot e_{\chi_{1}} \stackrel{\simeq}{\longrightarrow} \mathbb{C}E_{S}(K) \cdot e_{\chi_{1}}, $$
implies then that $\{\varepsilon_{i_{s}}\cdot \varepsilon_{i_{r+1}}^{-1} \cdot e_{\chi_{1}}\ \, | \, s =1,\ldots,r \}$ is a $\mathbb{C}$-basis for $\mathbb{C}E_{S}(K) \cdot e_{\chi_{1}}$.  It follows that a $\mathbb{C}$-basis for $\bigwedge^{r}_{\mathbb{C}[G]}\mathbb{C}E_{S}(K) \cdot e_{\chi_{1}}$ is given by $(\varepsilon_{i_{1}} \varepsilon_{i_{r+1}}^{-1}) \wedge \ldots \wedge (\varepsilon_{i_{r}} \varepsilon_{i_{r+1}}^{-1}) \cdot e_{\chi_{1}}$.  We calculate
\begin{align*}
&
\wedge^{r} f \circ \lambda_{K,S}\left((\varepsilon_{i_{1}} \varepsilon_{i_{r+1}}^{-1}) \wedge \ldots \wedge (\varepsilon_{i_{r}} \varepsilon_{i_{r+1}}^{-1}) \cdot e_{\chi_{1}} \right)
\\
&\qquad
=
\wedge^{r}f \left(\sum_{J \in \wp_{r}(\Omega)}R_{J,K}\left((\varepsilon_{i_{1}} \varepsilon_{i_{r+1}}^{-1}) \wedge \ldots \wedge (\varepsilon_{i_{r}} \varepsilon_{i_{r+1}}^{-1}) \cdot e_{\chi_{1}}\right)w_{J}   \right)
\\
&\qquad
=
\wedge^{r}f\left(R_{I,K}\left((\varepsilon_{i_{1}} \varepsilon_{i_{r+1}}^{-1}) \wedge \ldots \wedge(\varepsilon_{i_{r}} \varepsilon_{i_{r+1}}^{-1}) \right)(w_{i_{1}} - w_{i_{r+1}})\wedge \ldots \wedge(w_{i_{r}} - w_{i_{r+1}})  \cdot e_{\chi_{1}}\right) 
\\
&\qquad
=
\chi_{1}\left(R_{I,K}\left((\varepsilon_{i_{1}}\varepsilon_{i_{r+1}}^{-1})\wedge \ldots \wedge(\varepsilon_{i_{r}}\varepsilon_{i_{r+1}}^{-1}) \right) \right) \cdot (\varepsilon_{i_{1}} \varepsilon_{i_{r+1}}^{-1}) \wedge \ldots \wedge (\varepsilon_{i_{r}} \varepsilon_{i_{r+1}}^{-1}) \cdot e_{\chi_{1}}. 
\end{align*}
This completes the proof.
\end{proof}
\begin{Rem}
Proposition \ref{concretestarkreg} can be viewed as a generalization of \S $9$, Chapter I of \cite{Tate:1984} (in the abelian setting).
\end{Rem}

\subsection{Stark's conjecture over $\mathbb{Q}$} \label{mainst}
Now that we have decomposed the $S_{K}$-regulator $R_{K,S}$ into $\chi$-components, at least up to a rational number, the hope is that the decomposition (\ref{prod1}) would somehow match the decomposition (\ref{regulatorbroken}) and this is expressed in the following first conjecture of Stark (namely, Conjecture on page $61$ of \cite{Stark:1975} reformulated as Conjecture $5.1$ on page $27$ of \cite{Tate:1984}).  Note that the original conjecture was formulated for a general $S$-truncated Artin $L$-function, whereas we only treat the case where $K/k$ is an abelian extension of number fields.  But there is no loss in generality in doing so for Stark's conjecture over $\mathbb{Q}$ because of Proposition $7.2$ of \cite{Tate:1984}.  (Stark-type conjectures over $\mathbb{Z}$ in the non-abelian setting have only recently been formulated.  See, for instance, \cite{Burns:2011}.)
\begin{Con}[Stark's conjecture over $\mathbb{Q}$]
Let $\mathcal{A}$ be an Artin system of $S_{K}$-units.  For $\chi \in \widehat{G}$, we set
$$A(\chi,\mathcal{A}) = \frac{L_{K,S}^{*}(0,\chi)}{R(\chi,\mathcal{A})}. $$
Then
\begin{enumerate}
\item $A(\chi,\mathcal{A}) \in \overline{\mathbb{Q}}$,
\item $A(\chi,\mathcal{A})^{g} = A(\chi^{g},\mathcal{A})$ for all $g \in {\rm Gal}(\overline{\mathbb{Q}}/\mathbb{Q})$.
\end{enumerate}
\end{Con}

We shall now rephrase Stark's conjecture over $\mathbb{Q}$ in a slightly different way.  Let us define
$$\beta_{S}(\mathcal{A}) = \sum_{\chi \in \widehat{G}}A(\chi,\mathcal{A}) \cdot e_{\overline{\chi}}. $$
\begin{The} \label{reformulationQ}
Stark's conjecture over $\mathbb{Q}$ for all $\chi \in \widehat{G}$ is equivalent to
$$\beta_{S}(\mathcal{A}) \in \mathbb{Q}[G]. $$
\end{The}
\begin{proof}
  Assume first that Stark's conjecture over $\mathbb{Q}$ is true for all $\chi \in \widehat{G}$.  If $g \in {\rm Gal}(\overline{\mathbb{Q}}/\mathbb{Q})$, we have
\begin{equation*}
  \begin{aligned}
    \beta_{S}(\mathcal{A})^{g} &= \sum_{\chi \in \widehat{G}}A(\chi,\mathcal{A})^{g}\cdot e_{\overline{\chi}^{g}} \\
    &=\sum_{\chi \in \widehat{G}}A(\chi^{g},\mathcal{A}) \cdot e_{\overline{\chi}^{g}} \\
    &= \beta_{S}(\mathcal{A}).
  \end{aligned}
\end{equation*}
Hence, $\beta_{S}(\mathcal{A}) \in \mathbb{Q}[G]$.  Conversely, if we define $m_{\sigma}$ for $\sigma \in G$ via the equation
$$\beta_{S}(\mathcal{A}) = \sum_{\sigma \in G}m_{\sigma} \cdot \sigma^{-1}, $$
then the $m_{\sigma}$ and the $A(\chi,\mathcal{A})$ are related via the formulas
$$m_{\sigma} = \frac{1}{|G|}\sum_{\chi \in \widehat{G}} \overline{\chi(\sigma)} A(\chi,\mathcal{A}) $$
and
$$A(\chi,\mathcal{A}) = \sum_{\sigma \in G}\chi(\sigma) m_{\sigma}. $$
This last equation shows that if $\beta_{S}(\mathcal{A}) \in \mathbb{Q}[G]$, then $A(\chi,\mathcal{A}) \in \overline{\mathbb{Q}}$.  Moreover, if $\beta_{S}(\mathcal{A}) \in \mathbb{Q}[G]$, then $\beta_{S}(\mathcal{A})^{g} = \beta_{S}(\mathcal{A})$ for all $g \in {\rm Gal}(\overline{\mathbb{Q}}/\mathbb{Q})$.  But this last equation can be rewritten as
$$\sum_{\chi \in \widehat{G}}A(\chi^{g^{-1}},\mathcal{A})^{g} \cdot e_{\overline{\chi}} = \sum_{\chi \in \widehat{G}}A(\chi,\mathcal{A}) \cdot e_{\overline{\chi}}, $$
and this shows the desired result.
\end{proof}

\subsection{Popescu's conjecture} \label{popconj}
\emph{In this subsection, $r$ will stand for an integer satisfying $1 \le r \le |S|$}.  Popescu's conjecture concerns the $S$-truncated $L$-functions having minimal order of vanishing only, and is formulated under the following hypothesis.
\begin{Hyp} \label{StarkHhigher}
\hfill
\begin{enumerate}
\item The set $S$ contains $S_{\infty}$ and the places that ramify in $K/k$. \label{unhigher}
\item The set $S$ contains at least $r$ places that split completely in $K/k$, say $v_{1},\ldots,v_{r}$. \label{deuxhigher}
\item The set $S$ satisfies $|S| \ge r+1$. \label{troishigher}
\end{enumerate}
\end{Hyp}
Points $(\ref{deuxhigher})$ and $(\ref{troishigher})$ together with Theorem \ref{orderofvanishing} imply that ${\rm ord}_{s=0}L_{K,S}(s,\chi) \ge r$ for all $\chi \in \widehat{G}$.  From now on, we let
$$\widehat{G}_{r,S} = \{\chi \in \widehat{G}\, | \, \chi \neq \chi_{1} \text{ and } r_{S}(\chi) = r\}.$$
We also define
$$e_{r,S} =\sum_{\chi \in \widehat{G}_{r,S}}e_{\chi} \in \mathbb{Q}[G]. $$
Moreover, we set
\begin{equation*}
\widehat{G}_{r,S}' =
\begin{cases}
\widehat{G}_{r,S}, &\text{if } |S| \ge r+2 \\
\widehat{G}_{r,S} \cup \{\chi_{1} \}, &\text{if } |S| = r+1.
\end{cases}
\end{equation*}
and
\begin{equation*}
  e_{r,S}' = 
  \begin{cases}
    e_{r,S}, &\text{if } |S| \ge r+2 \\
    e_{r,S} + e_{\chi_{1}}, &\text{if } |S| = r+1.
  \end{cases}
\end{equation*}
Note that $e_{r,S}, e_{\chi_{1}}$, and $e_{r,S}' \in \mathbb{Q}[G]$.  Moreover if $S$ satisfies $(\ref{deuxhigher})$ of Hypothesis \ref{StarkHhigher} and $\chi \in \widehat{G}_{r,S}$, then $G_{1},\ldots,G_{r}$, which are trivial in this case, are the unique decomposition groups contained in ${\rm ker}(\chi)$ by Theorem \ref{orderofvanishing}.
\begin{Pro} \label{isohigher}
  \hfill
\begin{enumerate}
\item Assuming that the set $S$ satisfies $(\ref{deuxhigher})$ of Hypothesis \ref{StarkHhigher}, the $\mathbb{C}[G]$-linear morphism
$$R_{I,K}:\bigwedge_{\mathbb{C}[G]}^{r} \mathbb{C}E_{S}(K)\cdot e_{r,S} \longrightarrow \mathbb{C}[G] \cdot e_{r,S}, $$
where $I = (1,2,\ldots,r)$, is an isomorphism of $\mathbb{C}[G]$-modules. 
\item Assuming that $|S| = r+1$, then for any $J \in \wp_{r}(\Omega)$ the map
$$R_{J,K}:\bigwedge_{\mathbb{C}[G]}^{r}\mathbb{C}E_{S}(K) \cdot e_{\chi_{1}} \longrightarrow \mathbb{C}[G] \cdot e_{\chi_{1}} $$
is an isomorphism of $\mathbb{C}[G]$-modules.  Moreover, $R_{J_{1},K} = \pm R_{J_{2},K}$ on $\mathbb{C}E_{S}(K) \cdot e_{\chi_{1}}$ for any $J_{1},J_{2} \in \wp_{r}(\Omega)$.
\end{enumerate}
\end{Pro}
\begin{proof}
  For the first part, note that
  $${\rm dim}_{\mathbb{C}}\left(\bigwedge_{\mathbb{C}[G]}^{r}\mathbb{C}E_{S}(K) \cdot e_{r,S} \right) = |\widehat{G}_{r,S}|= {\rm dim}_{\mathbb{C}} \left(\mathbb{C}[G] \cdot e_{r,S} \right). $$
  It is therefore sufficient to show that $R_{I,K}$ is injective.  But if $R_{I,K}(x \cdot e_{r,S}) = 0$ for some $x \in \bigwedge_{\mathbb{C}[G]}^{r}\mathbb{C}E_{S}(K)$, then
  \begin{equation*}
    \begin{aligned}
      \wedge^{r}\lambda_{K,S}(x \cdot e_{r,S}) &= \sum_{J \in \wp_{r}(\Omega)}R_{J,K}(x \cdot e_{r,S}) \cdot w_{J}\\
      &=R_{I,K}(x \cdot e_{r,S}) \\
      &= 0,
    \end{aligned}
  \end{equation*}
  by Proposition \ref{isotyp}.  Since $\wedge^{r}\lambda_{K,S}$ is injective, we get that $x \cdot e_{r,S} = 0$ as we wanted to show.

  For the second part, a simple calculation using the product formula shows that $R_{J_{1},K}=\pm R_{J_{2},K}$ on $\bigwedge_{\mathbb{C}[G]}^{r}\mathbb{C}E_{S}(K)\cdot e_{\chi_{1}}$ for all $J_{1},J_{2} \in \wp_{r}(\Omega)$.  Now, we have again
  $${\rm dim}_{\mathbb{C}}\left(\bigwedge_{\mathbb{C}[G]}^{r}\mathbb{C}E_{S}(K)\cdot e_{\chi_{1}} \right) = 1 = {\rm dim}_{\mathbb{C}}(\mathbb{C}[G] \cdot e_{\chi_{1}}), $$
  and hence it is sufficient to show that $R_{J,K}$ is injective.  But if $R_{J,K}(x \cdot e_{\chi_{1}}) = 0$ for some $x \in \bigwedge_{\mathbb{C}[G]}^{r}\mathbb{C}E_{S}(K)$, then it follows that $R_{J',K}(x \cdot e_{\chi_{1}}) = 0$ for all $J' \in \wp_{r}(\Omega)$.  Therefore, by Proposition~\ref{isohigher} we have
  $$\wedge^{r}\lambda_{K,S}(x \cdot e_{\chi_{1}}) = \sum_{J' \in \wp_{r}(\Omega)}R_{J',K}(x \cdot e_{\chi_{1}}) \cdot w_{J'} = 0. $$
Since $\wedge^{r}\lambda_{K,S}$ is injective, this ends the proof.
\end{proof}
We now define some evaluators that are the main objects of study regarding Popescu's conjecture.
\begin{Def}
  \hfill
  \begin{enumerate}
  \item Assuming (\ref{deuxhigher}) of Hypothesis \ref{StarkHhigher}, we define the evaluator $\eta \in \bigwedge_{\mathbb{C}[G]}^{r}\mathbb{C}E_{S}(K) \cdot e_{r,S}$ to be the unique element of $\bigwedge_{\mathbb{C}[G]}^{r}\mathbb{C}E_{S}(K) \cdot e_{r,S}$ such that
    $$R_{I,K}(\eta) = \theta_{K,S}^{*}(0) \cdot e_{r,S}, $$
    where $I = (1,2,\ldots,r)$.
  \item Assuming that $|S|=r+1$, for $J \in \wp_{r}(\Omega)$, we define the evaluator $\delta_{J}$ to be the unique element of $\bigwedge_{\mathbb{C}[G]}^{r}\mathbb{C}E_{S}(K) \cdot e_{\chi_{1}}$ satisfying
    $$R_{J,K}(\delta_{J}) = \theta_{K,S}^{*}(0)\cdot e_{\chi_{1}}. $$
  \item Assuming (\ref{deuxhigher}) and (\ref{troishigher}) of Hypothesis \ref{StarkHhigher}, we let
    \begin{equation*}
      \eta' = 
      \begin{cases}
        \eta, &\text{if } |S| \ge r+2,\\
        \eta + \delta_{I}, &\text{if } |S| = r+1,
      \end{cases}
    \end{equation*}
    where $I = (1,2,\ldots,r)$.
  \end{enumerate}
\end{Def}
The uniqueness of these evaluators follow from Proposition \ref{isohigher}.  Moreover, $\eta'$ is the unique element of $\bigwedge_{\mathbb{C}[G]}^{r}\mathbb{C}E_{S}(K) \cdot e_{r,S}'$ satisfying
$$R_{I,K}(\eta') = \theta_{K,S}^{*}(0) \cdot e_{r,S}'. $$
The following proposition turns out to be important for us.
\begin{Pro} \label{formulaartin}
With the notation as above, if $S$ satisfies (\ref{deuxhigher}) and (\ref{troishigher}) of Hypothesis \ref{StarkHhigher}, and if $\mathcal{A} = \{\varepsilon_{w} \, | \, w \in S_{K} \}$ is an Artin system of $S_{K}$-units, then
\begin{equation*}
\eta' = 
\begin{cases}
\beta_{S}(\mathcal{A}) \cdot e_{r,S} \cdot \varepsilon_{1} \wedge \ldots \wedge \varepsilon_{r}, &\text{if } |S| \ge r+2 \\
\beta_{S}(\mathcal{A}) \cdot e_{r,S}' \cdot (\varepsilon_{1}\varepsilon_{r+1}^{-1}) \wedge \ldots \wedge(\varepsilon_{r}\varepsilon_{r+1}^{-1}), &\text{if } |S| = r+1.
\end{cases} 
\end{equation*}
\end{Pro}
\begin{proof}
  Let $\mathcal{A} = \{\varepsilon_{w}\, | \, w \in S_{K} \}$ be an Artin system of $S_{K}$-units.  Assuming first that $|S| = r+1$, and using Propositions \ref{unitschi} and \ref{concretestarkreg}, we calculate
  \begin{equation*}
    \begin{aligned}
      \theta_{K,S}^{*}(0) \cdot e_{r,S}' &= \sum_{\chi \in \widehat{G}_{r,S}} L_{K,S}^{*}(0,\chi) \cdot e_{\overline{\chi}} + L_{K,S}^{*}(0,\chi_{1}) \cdot e_{\chi_{1}} \\
      &= \sum_{\chi \in \widehat{G}_{r,S}}A(\chi,\mathcal{A}) R(\chi,\mathcal{A}) \cdot e_{\overline{\chi}} + A(\chi_{1},\mathcal{A})R(\chi_{1},\mathcal{A}) \cdot e_{\chi_{1}} \\
      &= R_{I,K}\left((\varepsilon_{1}\varepsilon_{r+1}^{-1})\wedge \ldots \wedge(\varepsilon_{r}\varepsilon_{r+1}^{-1}) \right)\beta_{S}(\mathcal{A}) \cdot e_{r,S}'.
    \end{aligned}
  \end{equation*}
  It follows that
  \begin{equation} 
    \eta' = \beta_{S}(\mathcal{A})\cdot e_{r,S}' \cdot (\varepsilon_{1}\varepsilon_{r+1}^{-1}) \wedge \ldots \wedge(\varepsilon_{r}\varepsilon_{r+1}^{-1}).
    \end{equation}
If $|S| > r+1$, the calculation is similar and left to the reader.
\end{proof}
As a corollary, we obtain:
\begin{Cor} \label{overQ}
  If Stark's conjecture over $\mathbb{Q}$ is true, then
  $$\eta' \in \mathbb{Q}\bigwedge_{\mathbb{Z}[G]}^{r}E_{S}(K) \simeq \bigwedge_{\mathbb{Q}[G]}^{r}\mathbb{Q}E_{S}(K). $$
\end{Cor}
\begin{proof}
This follows from Proposition \ref{formulaartin}, Theorem \ref{reformulationQ}, and the fact that $e_{r,S}' \in \mathbb{Q}[G]$.
\end{proof}
\begin{Rem}
Corollary \ref{overQ} is well-known, but has never been spelled out explicitly in terms of an Artin system of $S_{K}$-units.
See for instance Proposition $2.3$ of \cite{Rubin:1996}.
\end{Rem}

If $M$ is a $\mathbb{Z}[G]$-module, then we let $M^{*} = {\rm Hom}_{\mathbb{Z}[G]}(M,\mathbb{Z}[G])$; that is, $M^{*}$ is the dual of $M$ in the category of $\mathbb{Z}[G]$-modules.

If $\varphi \in M^{*}$, then for any integer $r \ge 1$ it induces a $\mathbb{Z}[G]$-module morphism
$$\tilde{\varphi}:\bigwedge_{\mathbb{Z}[G]}^{r}M \longrightarrow \bigwedge_{\mathbb{Z}[G]}^{r-1}M, $$
defined by
$$m_{1}\wedge \ldots \wedge m_{r} \mapsto \sum_{i=1}^{r}(-1)^{i+1}\varphi(m_{i})m_{1} \wedge \ldots \wedge m_{i-1} \wedge m_{i+1} \wedge \ldots \wedge m_{r}. $$
If $\varphi_{1},\ldots,\varphi_{k} \in M^{*}$, then iterating this process gives a $\mathbb{Z}[G]$-module morphism
$$ \bigwedge_{\mathbb{Z}[G]}^{k}M^{*} \longrightarrow {\rm Hom}_{\mathbb{Z}[G]}\left(\bigwedge_{\mathbb{Z}[G]}^{r}M,\bigwedge_{\mathbb{Z}[G]}^{r-k}M \right),$$
defined by $\varphi_{1} \wedge \ldots \wedge \varphi_{k} \mapsto \tilde{\varphi}_{k} \circ \ldots \circ \tilde{\varphi}_{1}$.  When $k=r-1$, we obtain a map
\begin{equation*}
  \bigwedge_{\mathbb{Z}[G]}^{r}M^{*} \longrightarrow {\rm Hom}_{\mathbb{Z}[G]}\left(\bigwedge_{\mathbb{Z}[G]}^{r}M,M \right)
  \,.
\end{equation*}
If $M$ is a $\mathbb{Z}[G]$-module, then we shall denote the natural map $M \longrightarrow \mathbb{Q}M$ by $m \mapsto \widetilde{m}$.  Moreover, we let
$$E_{S}(K)^{ab} = \{u \in E_{S}(K) \, | \, K(u^{1/w_{K}})/k \text{ is abelian} \}. $$
One can check that $E_{S}(K)^{ab}$ is a $\mathbb{Z}[G]$-submodule of $E_{S}(K)$.  In \cite{Popescu:2002}, Popescu defines the following lattice:
\begin{Def}
  With notation as above, we set
  $$\Lambda_{K,S}^{ab} = \left\{x \in \mathbb{Q}\bigwedge_{\mathbb{Z}[G]}^{r}E_{S}(K) \, \Big| \, \varphi_{1}\wedge\ldots\wedge\varphi_{r-1}(x) \in \widetilde{E_{S}(K)^{ab}}, \text{ for all } \varphi_{1},\ldots,\varphi_{r-1} \in E_{S}(K)^{*} \right\}. $$
\end{Def}
Moreover, he states the following conjecture:
\begin{Con}[Popescu] \label{popconj1}
  Assuming that Hypothesis \ref{StarkHhigher} is satisfied, one has
  $$w_{K} \cdot \eta' \in \Lambda_{K,S}^{ab}. $$
\end{Con}
\begin{Rem}
When $r=1$, one recovers Stark's abelian rank one conjecture (Conjecture $1$ of \cite{Stark:1980} or Conjecture $2.1$ on page $89$ of \cite{Tate:1984}), since $\Lambda_{K,S}^{ab} = \widetilde{E_{S}(K)^{ab}}$.  That is, if $K/k$ is a finite abelian extension of number fields such that Hypothesis \ref{StarkHhigher} is satisfied for $r=1$, then there exists an $S_{K}$-unit $\varepsilon_{0} \in E_{S}(K)$ satisfying
\begin{enumerate}
\item $e_{1,S}' \cdot \widetilde{\varepsilon_{0}} = \widetilde{\varepsilon_{0}}$ in $\mathbb{Q}E_{S}(K)$, 
\item $L_{K,S}^{*}(0,\chi) = - \frac{1}{w_{K}} \sum_{\sigma \in G}\chi(\sigma)\log|\varepsilon_{0}^{\sigma}|_{w_{1}}$ for all $\chi \in \widehat{G}_{1,S}'$, 
\item $K(\varepsilon_{0}^{1/w_{K}})/k$ is a finite abelian extension of number fields.  
\end{enumerate}
Such an $S_{K}$-unit is called a Stark unit and is unique up to a root of unity.  If necessary, see \S $3.8$ of \cite{Vallieres:2011} for a comparison between the various slightly different formulations of Stark's abelian rank one conjecture that one can find in the literature.
\end{Rem}
\begin{Rem}
If $|S| \ge r+2$, and $\mathcal{A} = \{\varepsilon_{w}\, | \, w \in S_{K} \}$ is an Artin system of $S_{K}$-units, then Proposition \ref{formulaartin} gives
$$\eta' = \eta = \beta_{S}(\mathcal{A}) \cdot e_{r,S} \cdot \varepsilon_{1}\wedge \ldots \wedge \varepsilon_{r}. $$
  Hence, Popescu's conjecture predicts that
  $$w_{K} \cdot \beta_{S}(\mathcal{A}) \cdot e_{r,S} \cdot \varepsilon_{1}\wedge \ldots \wedge \varepsilon_{r} \in \Lambda_{K,S}^{ab}. $$
  Note that if $\varphi_{1},\ldots,\varphi_{r-1} \in E_{S}(K)^{*}$, then
  $$\varphi_{1}\wedge \ldots \wedge \varphi_{r-1}(\varepsilon_{1} \wedge \ldots \wedge \varepsilon_{r}) \in E_{S}(K). $$
  Assuming Stark's conjecture over $\mathbb{Q}$, one expects
  $$w_{K} \cdot \beta_{S}(\mathcal{A}) \cdot e_{r,S} \in \mathbb{Q}[G]. $$
  Therefore, Stark's conjecture over $\mathbb{Q}$ and Popescu's conjecture together predict that for all $\varphi_{1},\ldots,\varphi_{r-1} \in E_{S}(K)^{*}$, there exists an $S_{K}$-unit $\varepsilon \in E_{S}(K)^{ab}$ (which depends on $\varphi_{1},\ldots,\varphi_{r-1}$ and $\mathcal{A}$) such that in $\mathbb{Q}E_{S}(K)$ one has
  $$w_{K} \cdot \beta_{S}(\mathcal{A}) \cdot e_{r,S} \cdot \varphi_{1}\wedge \ldots \wedge \varphi_{r-1}(\varepsilon_{1} \wedge \ldots \wedge \varepsilon_{r}) = \widetilde{\varepsilon}. $$
If $|S| = r+1$, one has a similar prediction, but with a slightly different formula for $\eta'$ as explained in Proposition \ref{formulaartin}.  This observation can be used to perform numerical verifications of Popescu's conjecture.  We explain this in more detail in \S \ref{numerical} below.
\end{Rem}
\begin{Rem} \label{generatordual}
Starting with the short exact sequence of $\mathbb{Z}[G]$-modules
$$1 \longrightarrow \mu(K) \longrightarrow E_{S}(K) \longrightarrow \widetilde{E_{S}(K)} \longrightarrow 1,$$
and applying the functor ${\rm Hom}_{\mathbb{Z}[G]}(\, \cdot \,,\mathbb{Z}[G])$, one gets an isomorphism of abelian groups
\begin{equation} \label{iso1}
  {\rm Hom}_{\mathbb{Z}[G]}(\widetilde{E_{S}(K)},\mathbb{Z}[G]) \stackrel{\simeq}{\longrightarrow} {\rm Hom}_{\mathbb{Z}[G]}(E_{S}(K),\mathbb{Z}[G]),
\end{equation}
since $\mathbb{Z}[G]$ is $\mathbb{Z}$-free and $\mu(K)$ is finite.  In the sequel, we will identify elements of $E_{S}(K)^{*}$ with elements of $\widetilde{E_{S}(K)}^{*}$ using this isomorphism.

Furthermore, we remind the reader that given a $\mathbb{Z}[G]$-module $M$, one has an isomorphism of abelian groups
\begin{equation} \label{iso2}
  {\rm Hom}_{\mathbb{Z}}(M,\mathbb{Z}) \stackrel{\simeq}{\longrightarrow} {\rm Hom}_{\mathbb{Z}[G]}(M,\mathbb{Z}[G])
\end{equation}
given by $f \mapsto \hat{f}$, where
$$\hat{f}(m) = \sum_{\sigma \in G} f(\sigma^{-1} \cdot m)\cdot \sigma. $$

Starting with a set of fundamental $S_{K}$-units $\eta_{1},\ldots,\eta_{t}$ for $E_{S}(K)$, we can consider the $\widetilde{\eta_{i}}^{*} \in {\rm Hom}_{\mathbb{Z}}(\widetilde{E_{S}(K)},\mathbb{Z})$ defined by
$$\widetilde{\eta_{i}}^{*}(\widetilde{\eta_{j}}) = \delta_{ij}, $$
where $\delta_{ij}$ is the Kronecker symbol.  Using the isomorphisms (\ref{iso1}) and (\ref{iso2}) above, one finds that
$$\Sigma =\{\widehat{\widetilde{\eta_{i}}^{*}} \, | \, i=1,\ldots,t \} $$
is a generating set for $E_{S}(K)^{*}$.  Therefore,
$$\Lambda_{K,S}^{ab} = \left\{x \in \mathbb{Q}\bigwedge_{\mathbb{Z}[G]}^{r}E_{S}(K) \, \Big| \, \varphi_{1}\wedge\ldots\wedge\varphi_{r-1}(x) \in \widetilde{E_{S}(K)^{ab}}, \text{ for all } \varphi_{1},\ldots,\varphi_{r-1} \in \Sigma \right\}.$$
Using this last remark, one can check that a given element $x \in \mathbb{Q}\bigwedge_{\mathbb{Z}[G]}^{r}E_{S}(K)$ lies in $\Lambda_{K,S}^{ab}$ in finitely many steps.
\end{Rem}

\subsection{Burns's conjecture} \label{burnsconj}
\emph{In this subsection, $r$ will stand for an integer satisfying $1 \le r \le |S|$}.  Burns's conjecture is formulated under the same hypotheses as Popescu's conjecture, namely Hypothesis \ref{StarkHhigher}.  We let
\begin{equation*}
  S_{r} = \{v_{1},\ldots,v_{r} \}.
\end{equation*}
We now specialize Conjecture $4.4.1$ of \cite{Burns:2011} to the abelian setting and to an $S$-situation rather than a $T$-modified version.
\begin{Con}[Burns] \label{burnscon}
With the same notation as above, for every $\phi \in {\rm Hom}_{\mathbb{Z}[G]}(E_{S}(K),X_{S}(K))$ one has
$$w_{K} \cdot \theta_{K,S}^{*}(0) \cdot e_{r,S}' \cdot {\rm det}_{\mathbb{C}[G]}(\lambda_{K,S}^{-1} \circ \phi_{\mathbb{C}}) \in \mathbb{Z}[G].$$
Moreover,
\begin{enumerate}
  \item One has
    $$w_{K} \cdot  \theta_{K,S}^{*}(0) \cdot e_{r,S}' \cdot {\rm det}_{\mathbb{C}[G]}(\lambda_{K,S}^{-1} \circ \phi_{\mathbb{C}}) \in {\rm Ann}_{\mathbb{Z}[G]}(Cl_{S}(K)), $$
  \item If $S'$ is any finite set of places of $k$ satisfying $S_{\infty} \cup S_{r} \subseteq S' \subseteq S$, then for any \label{burns_two}
    $$b \in \bigcup_{v \in S \smallsetminus S'}{\rm Ann}_{\mathbb{Z}[G]}\left(\mathbb{Z}[G/G_{v}] \right), $$  
    one has
    $$b \cdot w_{K} \cdot  \theta_{K,S}^{*}(0) \cdot e_{r,S}' \cdot {\rm det}_{\mathbb{C}[G]}(\lambda_{K,S}^{-1} \circ \phi_{\mathbb{C}}) \in {\rm Ann}_{\mathbb{Z}[G]}(Cl_{S'}(K)). $$
\end{enumerate}
\end{Con}
\begin{Rem}
If $r=0$, then 
$$\theta_{K,S}^{*}(0) \cdot e_{r,S}' \cdot {\rm det}_{\mathbb{C}[G]}(\lambda_{K,S}^{-1} \circ \phi_{\mathbb{C}}) = \theta_{K,S}(0), $$
and Brumer's classical conjecture on annihilation of class groups predicts that 
$$\theta_{K,S}(0) \in {\rm Ann}_{\mathbb{Z}[G]}(Cl(K)). $$
\end{Rem}
\begin{Rem}
Note that since $X_{S}(K)$ is $\mathbb{Z}$-free, we have an isomorphism of abelian groups
$${\rm Hom}_{\mathbb{Z}[G]}(\widetilde{E_{S}(K)},X_{S}(K)) \simeq {\rm Hom}_{\mathbb{Z}[G]}(E_{S}(K),X_{S}(K)), $$
so from now on, we will identify these two abelian groups.
\end{Rem}
\begin{Rem} If we start with an Artin system of $S_{K}$-units $\{\varepsilon_{w} \, | \, w \in S_{K} \}$ with induced $\mathbb{Z}[G]$-morphism
$$f:Y_{S}(K) \longrightarrow E_{S}(K),$$
then it induces an injective morphism of $\mathbb{Z}[G]$-modules
$$f:X_{S}(K) \hookrightarrow E_{S}(K). $$
Therefore, we get an isomorphism of $\mathbb{Q}[G]$-modules
$$f_{\mathbb{Q}}:\mathbb{Q}X_{S}(K) \stackrel{\simeq}{\longrightarrow} \mathbb{Q}E_{S}(K). $$
Letting
$$m = [E_{S}(K):\mu(K) \cdot f(X_{S}(K))], $$
it is simple to check that the inverse map $f_{\mathbb{Q}}^{-1}:\mathbb{Q}E_{S}(K) \longrightarrow \mathbb{Q}X_{S}(K)$ induces a morphism
$$m \cdot f_{\mathbb{Q}}^{-1}:\widetilde{E_{S}(K)} \longrightarrow X_{S}(K). $$
Letting $\phi = m \cdot f_{\mathbb{Q}}^{-1}$, one has
\begin{equation*}
\begin{aligned}
\theta_{K,S}^{*}(0) \cdot e_{r,S}' \cdot {\rm det}_{\mathbb{C}[G]}(\lambda_{K,S}^{-1} \circ \phi_{\mathbb{C}}) &= \frac{\theta_{K,S}^{*}(0)}{{\rm det}_{\mathbb{C}[G]}(\phi_{\mathbb{C}}^{-1} \circ \lambda_{K,S})} \cdot e_{r,S}' \\
&= m^{r} \frac{\theta_{K,S}^{*}(0)}{{\rm det}_{\mathbb{C}[G]}(f_{\mathbb{C}} \circ \lambda_{K,S})} \cdot e_{r,S}'\\
&= m^{r} \cdot \beta_{S}(\mathcal{A}) \cdot e_{r,S}'.
\end{aligned}
\end{equation*}
Therefore, a particular case of Burns's conjecture could be phrased as follows:
\begin{Con}[Burns] \label{burnsquestion}
  Let $K/k$ be a finite abelian extension of number fields with Galois group $G$ and let $S$ be a finite set of places of $k$ satisfying Hypothesis \ref{StarkHhigher}.  Given an Artin system of $S_{K}$-units $\mathcal{A}$, one has
$$w_{K} \cdot m^{r} \cdot \beta_{S}(\mathcal{A}) \cdot e_{r,S}' \in \mathbb{Z}[G].$$
Moreover,
\begin{enumerate}
\item One has \label{un_11}
$$w_{K} \cdot m^{r} \cdot \beta_{S}(\mathcal{A}) \cdot e_{r,S}' \in {\rm Ann}_{\mathbb{Z}[G]}(Cl_{S}(K)).$$
\item If $S'$ is any finite set of places of $k$ satisfying $S_{\infty} \cup S_{r} \subseteq S' \subseteq S$, then for any
$$b \in \bigcup_{v \in S \smallsetminus S'}{\rm Ann}_{\mathbb{Z}[G]}\left(\mathbb{Z}[G/G_{v}] \right), $$
one has
$$b \cdot w_{K} \cdot m^{r} \cdot \beta_{S}(\mathcal{A}) \cdot e_{r,S}' \in {\rm Ann}_{\mathbb{Z}[G]}(Cl_{S'}(K)). $$
\end{enumerate}

\end{Con}
\end{Rem}

\subsection{A simple example} \label{example}
In this section, we study in detail a simple example in the order of vanishing one situation.  Specifically, we take $k = \mathbb{Q}$ and $K = \mathbb{Q}(\sqrt{10})$, and we let $G = \langle \sigma \rangle$.  Note that $h_{K}=2$ and we set 
$$S = \{v_{1},v_{2},v_{3} \} =\{\infty,2,5 \}. $$ 
The primes $2$ and $5$ are ramified in $K/\mathbb{Q}$ and we let $\mathfrak{p}_{2}$ and $\mathfrak{p}_{5}$ be the prime ideals of $K$ that satisfy
$$(2) = \mathfrak{p}_{2}^{2} \text{ and } (5) = \mathfrak{p}_{5}^{2}. $$
One has $\mathfrak{p}_{2} = (2,\sqrt{10})$ and $\mathfrak{p}_{5} = (5,\sqrt{10})$.  Also,
$$(\sqrt{10}) = \mathfrak{p}_{2} \cdot \mathfrak{p}_{5}. $$
It follows that $h_{K,S} = 1$.  We list the places of $S_{K}$ as follows:
$$\{w_{1},w_{1}',w_{2},w_{3} \}, $$
where $w_{1}$ corresponds to the real embedding $\sqrt{10} \mapsto \sqrt{10}$, $w_{1}'$ to the real embedding $\sqrt{10} \mapsto - \sqrt{10}$, $w_{2}$ to the prime ideal $\mathfrak{p}_{2}$, and $w_{3}$ to the prime ideal $\mathfrak{p}_{5}$.  From now on, we let
$$u = 3 + \sqrt{10}. $$
Note that $u$ is a fundamental unit for $E(K)$.  We have $\widehat{G} = \{\chi_{1},\chi \}$, where $\chi$ is the unique non-trivial character of $K/\mathbb{Q}$.  Note that
$$r_{S}(\chi) = 1 \text{ and } r_{S}(\chi_{1}) = 2. $$

A simple calculation using formulas (\ref{classnumberformula}) and (\ref{prodformulaL}) of \S \ref{lfunctions} shows that
$$L_{K,S}^{*}(0,\chi) = h_{K} \cdot R_{K} = 2 \cdot \log|u|_{w_{1}} $$
and 
$$L_{K,S}^{*}(0,\chi_{1}) = -\frac{h_{\mathbb{Q},S} \cdot R_{\mathbb{Q},S}}{w_{\mathbb{Q}}} = -\frac{1}{2}R_{\mathbb{Q},S} = -\frac{1}{2}\log(2) \log(5). $$
Moreover a Stark unit for the data $(K/\mathbb{Q},S,v_{1},w_{1})$ is given by
$$\varepsilon_{0} = u^{-2}, $$
that is
$$L'_{K,S}(0,\psi) = -\frac{1}{2} \sum_{\rho \in G}\psi(\rho)\cdot\log|\varepsilon_{0}^{\rho}|_{w_{1}}, $$
for all $\psi \in \widehat{G}$.  (For details, see for instance Proposition $3.13$ of \cite{Vallieres:2011}.)

From the calculations above follow that a fundamental system of $S_{K}$-units for $E_{S}(K)$ is given by
$$\{3 + \sqrt{10},2,\sqrt{10} \} = \{u,2,\sqrt{10} \}. $$
Following the proof of Theorem \ref{artinexist}, one finds the $S_{K}$-units
\begin{enumerate}
\item $\beta_{1} = (3 + \sqrt{10}) \cdot \sqrt{10} = 3\sqrt{10} + 10$,
\item $\beta_{2} = 2^{-2}  \cdot \sqrt{10} = \sqrt{10}/4$,
\item $\beta_{3} = 2  \cdot \sqrt{10}^{-1} = 2/\sqrt{10}$,
\end{enumerate}
that satisfy for $i=1,2,3$,  $|\beta_{i}|_{w_{i}}>1$ and $|\beta_{i}|_{w}<1$ for all $w \neq w_{i}$.  These $S_{K}$-units lead to the Artin system of $S_{K}$-units $\mathcal{A} = \{\varepsilon_{w} \, | \, w \in S_{K} \}$ where
\begin{enumerate}
\item $\varepsilon_{w_{1}} = 190 + 60 \sqrt{10}$,
\item $\varepsilon_{w_{1}'} = 190 - 60 \sqrt{10}$,
\item $\varepsilon_{w_{2}} = 25/64$,
\item $\varepsilon_{w_{3}} = 16/625$.
\end{enumerate}
Note that $\varepsilon_{w}^{\sigma} = \varepsilon_{w^{\sigma}}$ for all $w \in S_{K}$, and
$$\prod_{w \in S_{K}}\varepsilon_{w} = 1. $$
Hence, the kernel of the map $f:Y_{S}(K) \longrightarrow E_{S}(K)$ defined by $w \mapsto \varepsilon_{w}$ is $\mathbb{Z} \cdot \alpha$, where
$$\alpha = \sum_{v \in S}T_{v}(K). $$
To simplify the notation, we let (as we have done throughout) $\varepsilon_{i} = \varepsilon_{w_{i}}$ for $i=1,2,3$.  Now, using Proposition \ref{concretestarkreg}, we calculate
$$R(\chi,\mathcal{A}) = \log \Big|\frac{\varepsilon_{1}^{\sigma}}{\varepsilon_{1}} \Big|_{w_{1}}, $$
and
\begin{equation*}
R(\chi_{1},\mathcal{A}) = {\rm det}
\begin{pmatrix}
\log|N_{K/\mathbb{Q}}(\varepsilon_{1}\varepsilon_{3}^{-1})|_{v_{1}} & \log|N_{K/\mathbb{Q}}(\varepsilon_{2}\varepsilon_{3}^{-1})|_{v_{1}} \\
\log|N_{K/\mathbb{Q}}(\varepsilon_{1}\varepsilon_{3}^{-1})|_{v_{2}} & \log|N_{K/\mathbb{Q}}(\varepsilon_{2}\varepsilon_{3}^{-1})|_{v_{2}}
\end{pmatrix}.
\end{equation*}
Using the fact that $\varepsilon_{1} = 10u^{2}$, a simple calculation shows that 
$$\frac{L_{K,S}^{*}(0,\chi)}{R(\chi,\mathcal{A})} = -\frac{1}{2} $$
and
$$\frac{L_{K,S}^{*}(0,\chi_{1})}{R(\chi_{1},\mathcal{A})} = -\frac{1}{256}. $$
Therefore
$$\beta_{S}(\mathcal{A}) = \frac{1}{512} \cdot (-129 + 127 \sigma) \in \mathbb{Q}[G] $$
as predicted by Stark's conjecture over $\mathbb{Q}$.  (See Theorem \ref{reformulationQ}.)  
Note that
$$e_{1,S}' = e_{1,S} = \frac{1}{2}\left(1 - \sigma \right) \in \mathbb{Q}[G], $$
and thus
$$\beta_{S}(\mathcal{A}) \cdot e_{1,S} = -\frac{1}{4}(1-\sigma) \in \mathbb{Q}[G]. $$
Proposition \ref{formulaartin} then shows that
$$\eta' =  \eta = \beta_{S}(\mathcal{A}) \cdot e_{1,S} \cdot \varepsilon_{1}. $$
Now, Stark's abelian rank one conjecture, namely Conjecture \ref{popconj1} when $r=1$, predicts that
$$\widetilde{\varepsilon_{0}} = 2 \cdot \beta_{S}(\mathcal{A}) \cdot e_{1,S} \cdot \varepsilon_{1} $$
in $\mathbb{Q}E_{S}(K)$.  In other words, we should have 
$$\varepsilon_{0}^{2} = \varepsilon_{1}^{\sigma -1} $$
in $E_{S}(K)$ up to a root of unity in $K$ (that is $\pm 1$).  But this is indeed the case as a simple calculation shows.

We calculate furthermore
$$m = [E_{S}(K):\mu(K) \cdot f(X_{S}(K))] = 256. $$
Hence, we have
$$w_{K} \cdot m \cdot \beta_{S}(\mathcal{A}) \cdot e_{1,S} \in \mathbb{Z}[G] $$
as predicted by Conjecture \ref{burnsquestion}.  The annihilation part of Conjecture \ref{burnsquestion} is obviously satisfied, since $h_{K,S} = 1$, $h_{K} = 2$ and $w_{K} \cdot m \cdot \beta_{S}(\mathcal{A}) \cdot e_{1,S}  \in 2 \cdot \mathbb{Z}[G]$.

\section{Numerical calculations} \label{numerical} 

\subsection{The algorithm} \label{algorithm}
Let $k$ be a real quadratic field, and let $K$ be a cubic extension of $k$ that is totally real and such that $K/k$ is ramified.  We let $S$ be the set of places of $k$ consisting of the two archimedean places and the finite primes that ramify in $K/k$.  Hence, we always have $|S| \ge 3$.  We let
$$S = \{v_{1},v_{2},\ldots,v_{n}\}, $$
where we agree that $v_{1}$ and $v_{2}$ are the two archimedean places.  Note that $v_{1}$ and $v_{2}$ split completely in $K/k$, since $K$ is assumed to be totally real.  We now explain how to numerically verify Stark's conjecture over $\mathbb{Q}$, the rank two Popescu conjecture, and Burns's conjecture in this particular case.  All the calculations have been done with the software PARI (\cite{PARI}).
\begin{Step}
We calculate a fundamental system of $S_{K}$-units for $E_{S}(K)$, say $\{\eta_{1},\ldots,\eta_{t} \}$.
\end{Step}
\begin{Step}
For each $v_{i} \in S$ ($i=1,\ldots,n$), we choose a place $w_{i}$ lying above $v_{i}$.
\end{Step}
\begin{Step}
We calculate an Artin system of $S_{K}$-units $\mathcal{A} = \{\varepsilon_{w} \, | \, w \in S_{K} \}$.  Here, we follow the proof of Theorem \ref{artinexist} and the main step is to find $S_{K}$-units $\beta_{i}$ that satisfy $|\beta_{i}|_{w_{i}} > 1$ and $|\beta_{i}|_{w}< 1$ for all $w \in S_{K}$ satisfying $w \neq w_{i}$.  In order to find these $S_{K}$-units, we proceed as follows.  We consider the matrix
$$A = (\log|\eta_{j}|_{w}) \in M_{t+1,t}(\mathbb{R}), $$
where $t = |S_{K}| - 1$, and for $s \in \{1,\ldots,t+1 \}$, we let $A_{s}$ be the matrix obtained from $A$ by removing the $s$th row.  The matrices $A_{s}$ are $t \times t$ square matrices.  Furthermore, we let
$$\omega = (-1,\ldots,-1) \in M_{1,t}(\mathbb{R}). $$
Now, if we want to find $\beta_{i}$ then we look at $A_{s}$, where $s$ corresponds to the row involving the place $w_{i}$ and we calculate
$$x = A_{s}^{-1} \cdot \omega^{t}. $$
We then round off the coordinates of $x^{t}$ to the nearest integer in order to get a vector $y = (y_{1},\ldots,y_{n}) \in M_{1,t}(\mathbb{Z})$ and we set
$$\beta_{i} = \prod_{\ell=1}^{t}\eta_{\ell}^{y_{\ell}} \in E_{S}(K). $$
We check that $\beta_{i}$ satisfies $|\beta_{i}|_{w} < 1$ for all $w \neq w_{i}$.  If not, we repeat the process above with $n_{0} \cdot \omega$ where $n_{0}$ is a positive integer and we keep increasing $n_{0}$ until we find a $\beta_{i}$ with the desired properties.  The last condition $|\beta_{i}|_{w_{i}} > 1$ is automatically satisfied by the product formula (\ref{prodformula}). 
\end{Step}
\begin{Step}
Using Proposition \ref{concretestarkreg} and the PARI command \emph{bnrL1}, we calculate 
$$\beta_{S}(\mathcal{A}) \cdot e_{2,S}' = \sum_{\chi \in \widehat{G}_{2,S}'} A(\chi,\mathcal{A}) \cdot e_{\overline{\chi}} $$
to a high precision.
\end{Step}
\begin{Step}
Since $e_{2,S}' \in \mathbb{Q}[G]$, Stark's conjecture over $\mathbb{Q}$ via Theorem \ref{reformulationQ} predicts that
$$\beta_{S}(\mathcal{A}) \cdot e_{2,S}' = \sum_{\sigma \in G}b_{\sigma} \cdot \sigma \in \mathbb{Q}[G]. $$
Using the PARI command \emph{algdep}, we recognize the numbers $b_{\sigma}$ as rational numbers.
\end{Step}

\begin{Step}
We find the smallest positive integer $d$ such that
$$d \cdot \beta_{S}(\mathcal{A}) \cdot e_{2,S}' \in \mathbb{Z}[G]. $$
\end{Step}
\begin{Step}
We calculate 
$$m = [E_{S}(K):\mu(K) \cdot f(X_{S}(K))]. $$
If Conjecture \ref{burnsquestion} had a positive answer, then one would have $d \, | \, 2 m^{2}$ (since $w_{K} = 2$).  In fact, in all the examples that we computed, we observed numerically that $d \, | \, 2m$.  
\end{Step}
\begin{Step}
As explained in Remark \ref{generatordual}, we calculate for $i=1,\ldots,t$ the morphisms $\widehat{\widetilde{\eta_{i}}^{*}}$.
\end{Step}

\begin{Step}
For $i=1,\ldots,t$, we calculate $u_{i} \in E_{S}(K)$ where
\begin{equation*}
\widetilde{u_{i}} = 
\begin{cases}
\widehat{\widetilde{\eta_{i}}^{*}}(\varepsilon_{1}\wedge\varepsilon_{2}), &\text{if } |S|\ge 4 \\
\widehat{\widetilde{\eta_{i}}^{*}}((\varepsilon_{1}\varepsilon_{3}^{-1})\wedge(\varepsilon_{2}\varepsilon_{3}^{-1})), &\text{if } |S| = 3.
\end{cases}
\end{equation*}
Using Proposition \ref{formulaartin}, Popescu's conjecture is true if and only if 
$$w_{K} \cdot \beta_{S}(\mathcal{A}) \cdot e_{2,S}' \cdot \widetilde{u_{i}} \in \widetilde{E_{S}(K)^{ab}}$$
for all $i=1,\ldots,t$.  (Here $w_{K} = 2$, since $K$ is totally real.)  We can check this as follows.  First, we calculate the $S_{K}$-units $\gamma_{i}$ satisfying
$$\widetilde{\gamma_{i}} = 2 \cdot d \cdot \beta_{S}(\mathcal{A}) \cdot e_{2,S}' \cdot \widetilde{u_{i}}.$$
\end{Step}
\begin{Step}
Then, we find $S_{K}$-units $\delta_{i}$ such that we have $\widetilde{\gamma_{i}} = \widetilde{\delta_{i}^{d}}$.  These $S_{K}$-units satisfy
$$\widetilde{\delta_{i}} = 2 \cdot  \beta_{S}(\mathcal{A}) \cdot e_{2,S}' \cdot \widetilde{u_{i}}. $$
\end{Step}

\begin{Step}
Finally, we check that the extension $K(\delta_{i}^{1/2})/k$ is abelian, for $i=1,\ldots,t$.  In order to do so, we use the following well-known lemma:
\begin{Lem}
With the setup as above, let $\lambda \in K^{\times}$ and let $\sigma$ be a generator for $G$.  Then $K(\lambda^{1/2})/k$ is abelian if and only if $\lambda^{\sigma - 1} \in (K^{\times})^{2}$.
\end{Lem}
\begin{proof}
See Lemma $4.33$ of \cite{Vallieres:2011} for details if needed.
\end{proof}
\end{Step}
\begin{Step}
Given an element $\alpha \in \mathbb{Z}[G]$, one can check that $\alpha \in {\rm Ann}_{\mathbb{Z}[G]}(Cl(K))$ as follows.  Pick generators $[\mathfrak{a}_{1}],\ldots,[\mathfrak{a}_{h}]$ for $Cl(K)$ and check that $\mathfrak{a}_{i}^{\alpha}$ is a principal ideal for all $i=1,\ldots,h$.  A similar procedure also allows one to check that $\alpha \in {\rm Ann}_{\mathbb{Z}[G]}(Cl_{S}(K))$.  This allows us to check the annihilation statement of Conjecture \ref{burnsquestion}.
\end{Step}

\subsection{Computational results} \label{data}

Let $\mathcal{F}$ denote the collection of all totally real number fields $K$ such that $K/k$ is a ramified abelian extension and $k$ is a real quadratic field.  Our aim is to run the algorithm on all fields $K\in\mathcal{F}$ with $\Delta_{K}\leq 10^{12}$.
By the standard formula for discriminants in towers we know that
\[
  \Delta_{K}=\Delta_{k}^3\cdot \mathbb{N}(\Delta_{K/k})
  \,,
\]
and by the conductor-discriminant formula
(see, for example, Corollary~2 of~\cite{Artin/Tate:2009}) we know that $\Delta_{K/k}=\mathfrak{f}^2$ where $\mathfrak{f}$ is the conductor of $K/k$.  Consequently, we have:
\[
  \Delta_{K}\leq X
  \;\Longleftrightarrow\;
  \begin{cases}
  \Delta_{k}\leq X^{1/3}\\
  \mathbb{N}(\mathfrak{f})\leq \sqrt{\frac{X}{\Delta_{k}^3}}
  \end{cases}
\]
Hence to enumerate all fields in $\mathcal{F}$ up to discriminant $10^{12}$ it suffices to consider only real quadratic fields
$k$ with $\Delta_{k}\leq 10^4/\sqrt[3]{4} \approx 6300$, since $\mathbb{N}(\mathfrak{f}) \ge 2$.

For each such real quadratic field $k$, we iterate through all ideals $\mathfrak{a}$ of $k$ with
$1<\mathbb{N}(\mathfrak{a})\leq 10^6\cdot \Delta_{k}^{-3/2}$.
For each such~$\mathfrak{a}$, we
locate all the cubic subfields $K$ of the ray class field $k_\mathfrak{a}$
that satisfy $\mathfrak{f}(K/k)=\mathfrak{a}$, where $\mathfrak{f}(K/k)$ is the conductor of $K/k$, if any exist.

In turns out that there are $581$ real quadratic fields $k$ for which there is at least one ramified abelian cubic extension $K$ with $\Delta_{K}\le 10^{12}$.  The largest square-free integer $d$ for which $\mathbb{Q}(\sqrt{d})$ has such a cubic extension is $d = 3853$.

For each such extension $K/k$, we perform the algorithm presented in \S \ref{algorithm} for a total of $19197$ examples.  These calculations took $58.7$ (one-core) CPU hours on an Intel Xeon Haswell 3.20 GHz processor with eight cores.

\subsubsection{Popescu's conjecture}
There are three different cases that arise for our cubic extensions $K/k$:
\begin{enumerate}
\item $K/\mathbb{Q}$ is abelian, \label{ununun}
\item $K/\mathbb{Q}$ is Galois, but not abelian, \label{deuxdeuxdeux}
\item $K/\mathbb{Q}$ is not Galois. \label{troistroistrois}
\end{enumerate}
We list the number fields encountered in each case according to their class number in Table \ref{table_one} below.
As explained before, case (\ref{ununun}) is known by previous results of Burns, but we have performed the calculations for the sake of completeness.
We now explain one example in detail.  Our algorithm completes the calculation for this particular extension of number fields in $3.7$ seconds.

Take $k = \mathbb{Q}(\sqrt{3})$.  The rational prime $3$ is ramified in $k$ whereas $5$ is inert in $k$.  Let $\mathfrak{p}$ be the unique prime ideal of $k$ lying above $3$ and $\mathfrak{q}$ be the unique prime ideal lying above $5$.  Let $\mathfrak{m} = \mathfrak{p} \mathfrak{q}$ and consider the ray class field $k_{\mathfrak{m}}$.  One has $[k_{\mathfrak{m}}:k] = 12$ and there is a unique subfield $K$ that is a cubic abelian extension of $k$.  It has class number $1$.  The field $K$ is Galois over $\mathbb{Q}$, but its Galois group is not abelian.  A defining polynomial for $K$ is given by
$$p(x) = x^{6} - 24x^{4} - 50x^{3} - 3x^{2} + 30x - 2. $$
Both $\mathfrak{p}$ and $\mathfrak{q}$ are ramified in $K$, so $|S| = 4$ and $|S_{K}| = 8$.  We now go through the steps presented in \S \ref{algorithm}.

\noindent\textbf{Step 1.}
A fundamental system of $S_{K}$-units is given by the following polynomials modulo $(p(x))$:
\begin{enumerate}
\item
$\frac{7}{68} \, x^{5} - \frac{5}{68} \, x^{4} - \frac{145}{68} \, x^{3} - \frac{295}{68} \, x^{2} - \frac{40}{17} \, x + \frac{133}{34}$
\item
$\frac{39}{68} \, x^{5} - \frac{57}{68} \, x^{4} - \frac{837}{68} \, x^{3} - \frac{779}{68} \, x^{2} + \frac{207}{17} \, x - \frac{7}{34}$
\item
$\frac{11}{34} \, x^{5} - \frac{3}{34} \, x^{4} - \frac{257}{34} \, x^{3} - \frac{483}{34} \, x^{2} + \frac{3}{17} \, x + \frac{124}{17}$
\item
$\frac{13}{34} \, x^{5} - \frac{19}{34} \, x^{4} - \frac{279}{34} \, x^{3} - \frac{237}{34} \, x^{2} + \frac{138}{17} \, x - \frac{8}{17}$
\item
$\frac{7}{17} \, x^{5} - \frac{22}{17} \, x^{4} - \frac{111}{17} \, x^{3} + \frac{28}{17} \, x^{2} + \frac{112}{17} \, x - \frac{57}{17}$
\item
$-\frac{15}{68} \, x^{5} + \frac{1}{68} \, x^{4} + \frac{369}{68} \, x^{3} + \frac{671}{68} \, x^{2} - \frac{9}{17} \, x - \frac{81}{34}$
\item
$\frac{13}{68} \, x^{5} - \frac{19}{68} \, x^{4} - \frac{279}{68} \, x^{3} - \frac{305}{68} \, x^{2} + \frac{35}{17} \, x + \frac{43}{34}$
\end{enumerate}

\noindent\textbf{Step 2.}
The eight places in $S_{K}$ are
$$\{w_{1},w_{1}',w_{1}'',w_{2},w_{2}',w_{2}'',w_{3},w_{4}\} = \{-2.873, 0.620, 5.716, -2.233, -1.297, 0.067, \mathfrak{P}_{3},\mathfrak{P}_{5}\},$$
where $\mathfrak{P}_{3}$ is the unique finite prime lying above $\mathfrak{p}_{3}$ (similarly for $\mathfrak{P}_{5}$), and the floating-point numbers $\xi$ correspond to the real embeddings $x \mapsto \xi$.

\noindent\textbf{Step 3.}
The matrix $A$ is given by
\begin{equation*}
A=
{\scriptsize
\begin{pmatrix}
-1.316958 & -1.316958 & -1.316958 & 1.316958 & 1.316958 & 1.316958 & 0.0000000 & 0.0000000 \\
1.979440 & -2.556268 & 0.5768282 & -1.979440 & 2.556268 & -0.5768282 & 0.0000000 & 0.0000000 \\
-0.5768282 & -1.979440 & 2.556268 & -2.556268 & 0.5768282 & 1.979440 & 0.0000000 & 0.0000000 \\
-1.065300 & -2.054417 & 3.119716 & 1.065300 & 2.054417 & -3.119716 & 0.0000000 & 0.0000000 \\
3.119716 & -1.065300 & -2.054417 & 2.054417 & -3.119716 & 1.065300 & 0.0000000 & 0.0000000 \\
1.365299 & 0.8634475 & -1.679440 & -1.679440 & 1.365299 & 0.8634475 & -1.098612 & 0.0000000 \\
-0.1781805 & -1.655769 & 3.443387 & 0.8871190 & 0.3986478 & 0.3236711 & 0.0000000 & -3.218876
\end{pmatrix},
}
\end{equation*}
and we found the $S_{K}$-units $\beta_{i}$, ($i=1,2,3,4$) whose coordinates on the system of fundamental $S_{K}$-units are given by 
\begin{enumerate}
\item
$\beta_1=[-3, 2, -1, 1, 3, 3, 1]$,
\item
$\beta_2=[2, -3, -1, 3, 2, 2, 1]$,
\item
$\beta_3=[0, 9, -5, -9, -4, -13, 1]$,
\item
$\beta_4=[0, -3, 4, 4, 2, 2, -4]$.
\end{enumerate}

These four $S_{K}$-units lead to the following Artin system of $S_{K}$-units, given as polynomials modulo $(p(x))$:
\begin{enumerate}
\item $\varepsilon_{w_{1}} = -\frac{745941483}{68}x^{5} + \frac{2143283961}{68}x^{4} + \frac{11744383129}{68}x^{3} + \frac{3552405747}{68}x^{2} - \frac{1992290385}{17}x + \frac{259615015}{34} $
\item $\varepsilon_{w_{1}'} = -\frac{614733423}{34}x^5 - \frac{381601623}{34}x^4 + \frac{14516719313}{34}x^3 + \frac{39748062825}{34}x^2 + \frac{13259094279}{17}x - \frac{990292384}{17} $
\item $\varepsilon_{w_{1}''} = \frac{4808859}{68}x^{5} + \frac{27490335}{68}x^{4} + \frac{41738695}{68}x^{3} - \frac{1839447}{68}x^{2} - \frac{6235494}{17}x + \frac{841213}{34} $
\item $\varepsilon_{w_{2}} = -\frac{4649005}{17}x^{5} + \frac{10385828}{17}x^{4} + \frac{88374291}{17}x^{3} + \frac{35023017}{17}x^{2} - \frac{64294053}{17}x + \frac{4162067}{17}$
\item $\varepsilon_{w_{2}'}= \frac{21452083}{68}x^{5} - \frac{27839861}{68}x^{4} - \frac{478720265}{68}x^{3} - \frac{451335551}{68}x^{2} + \frac{130343317}{17}x - \frac{16529959}{34}$
\item $\varepsilon_{w_{2}''} = -\frac{20133813}{68}x^{5} - \frac{1362201}{68}x^{4} + \frac{483119351}{68}x^{3} + \frac{1039377233}{68}x^{2} + \frac{32680736}{17}x - \frac{297585015}{34}$
\item $\varepsilon_{w_{3}} = \frac{35}{148716}x^{5} - \frac{25}{148716}x^{4} - \frac{725}{148716}x^{3} - \frac{1475}{148716}x^{2} - \frac{200}{37179}x + \frac{325}{74358} $
\item $\varepsilon_{w_{4}} = \frac{3}{625}$
\end{enumerate}
The kernel of the induced $\mathbb{Z}[G]$-module morphism $f: Y_{S}(K) \longrightarrow E_{S}(K)$ is given by $\mathbb{Z} \cdot \alpha$, where
$$\alpha = 100 \cdot T_{v_{1}}(K) + 150 \cdot T_{v_{2}}(K) + 58 \cdot T_{v_{3}}(K) + 77 \cdot T_{v_{4}}(K). $$

\noindent\textbf{Step 4.}  We obtained
$$\beta_{S}(\mathcal{A}) \cdot e_{2,S}' = 0.030868\cdot {\rm id}-0.013639\cdot \sigma_{1} -0.017229\cdot \sigma_{2},$$
where the Galois automorphisms are given by
\begin{enumerate}
\item ${\rm id}: x \mapsto x$,
\item $\sigma_{1}: x \mapsto -\frac{9}{68}x^{5} + \frac{21}{68}x^{4} + \frac{167}{68}x^{3} + \frac{83}{68}x^{2} - \frac{53}{17}x - \frac{69}{34}$,
\item $\sigma_{2}: x \mapsto  -\frac{5}{68}x^{5} - \frac{11}{68}x^{4} + \frac{123}{68}x^{3} + \frac{507}{68}x^{2} + \frac{116}{17}x - \frac{61}{34}$.
\end{enumerate}

\noindent\textbf{Step 5.}  After recognizing the rational numbers, we obtained

$$\beta_{S}(\mathcal{A}) \cdot e_{2,S}' = \frac{43}{1393} \cdot {\rm id}-\frac{19}{1393}\cdot \sigma_{1}-\frac{24}{1393}\cdot \sigma_{2} \in \mathbb{Z}[G]. $$

\noindent\textbf{Step 6.}  Thus $d = 1393 = 7 \cdot 199$.

\noindent\textbf{Step 7.}  On the other hand, we obtained 
$$m = [E_{S}(K):\mu(K) \cdot f(X_{S}(K))] = 3698415 = 3^{2} \cdot 5 \cdot 7 \cdot 59 \cdot 199.$$
Note that $d \, | \, 2 m$.  (In fact, $d \, | \, m$ here, but there are cases where $d \nmid m$, $d \, | \, 2m$.)

\noindent\textbf{Step 8.}  The morphism $\widehat{\widetilde{\eta_{i}}^{*}} \in {\rm Hom}_{\mathbb{Z}[G]}(E_{S}(K),\mathbb{Z}[G]) \simeq M_{3,7}(\mathbb{Z})$ are given by

\begin{enumerate}
\item
$\widehat{\widetilde{\eta_1}^*}=[[1, 0, 0, 0, 0, 0, 0], [1, 0, 0, 0, 0, 0, 0], [1, 0, 0, 0, 0, 0, 0]]$,
\item
$\widehat{\widetilde{\eta_2}^*}=[[0, 1, 0, 0, 0, 0, 0], [0, -1, -1, 0, 0, -1, 0], [0, 0, 1, 0, 0, -1, 1]]$,
\item
$\widehat{\widetilde{\eta_3}^*}=[[0, 0, 1, 0, 0, 0, 0], [0, 1, 0, 0, 0, 1, -1], [0, -1, -1, 0, 0, 0, -1]]$,
\item
$\widehat{\widetilde{\eta_4}^*}=[[0, 0, 0, 1, 0, 0, 0], [0, 0, 0, -1, 1, 1, -1], [0, 0, 0, 0, -1, 1, -1]]$,
\item
$\widehat{\widetilde{\eta_5}^*}=[[0, 0, 0, 0, 1, 0, 0], [0, 0, 0, -1, 0, 1, -1], [0, 0, 0, 1, -1, 0, 0]]$,
\item
$\widehat{\widetilde{\eta_6}^*}=[[0, 0, 0, 0, 0, 1, 0], [0, 0, 0, 0, 0, 1, 0], [0, 0, 0, 0, 0, 1, 0]]$,
\item
$\widehat{\widetilde{\eta_7}^*}=[[0, 0, 0, 0, 0, 0, 1], [0, 0, 0, 0, 0, 0, 1], [0, 0, 0, 0, 0, 0, 1]]$.
\end{enumerate}

\noindent\textbf{Step 9.}  The coordinates of the $\gamma_{i}$ on the fundamental $S_{K}$-units are given by
\begin{enumerate}
\item
$\widetilde{\gamma_1}=[0, 0, 0, 0, 0, 0, 0]$,
\item
$\widetilde{\gamma_2}=[0, 0, 0, 2786, 2786, 0, 0]$,
\item
$\widetilde{\gamma_3}=[0, 0, 0, -2786, 0, 0, 0]$,
\item
$\widetilde{\gamma_4}=[0, -2786, 2786, 0, 0, 0, 0]$,
\item
$\widetilde{\gamma_5}=[0, -2786, 0, 0, 0, 0, 0]$,
\item
$\widetilde{\gamma_6}=[0, 0, 0, 0, 0, 0, 0]$,
\item
$\widetilde{\gamma_7}=[0, 0, 0, 0, 0, 0, 0]$.
\end{enumerate}
They are all divisible by $d = 1393$, as expected.

\noindent\textbf{Step 10.}  The coordinates of the $\delta_{i}$ on the fundamental $S_{K}$-units are 
\begin{enumerate}
\item $\delta_{1} = [0, 0, 0, 0, 0, 0, 0]$,
\item $\delta_{2} = [0, 0, 0, 2, 2, 0, 0]$,
\item $\delta_{3} = [0, 0, 0, -2, 0, 0, 0]$,
\item $\delta_{4} = [0, -2, 2, 0, 0, 0, 0]$,
\item $\delta_{5} = [0, -2, 0, 0, 0, 0, 0]$,
\item $\delta_{6} = [0, 0, 0, 0, 0, 0, 0]$,
\item $\delta_{7} = [0, 0, 0, 0, 0, 0, 0]$.
\end{enumerate}

\noindent\textbf{Step 11.}
The abelian condition is obviously satisfied in this case.  (We note that we did find examples where the units $\delta_{i}$ are not necessarily squares modulo roots of unity.)

\noindent\textbf{Step 12.}
Burns's conjecture is trivially true in this case since $h_K=1$.

\subsubsection{Burns's conjecture}
Recall that $d$ is the smallest positive integer satisfying
$d \cdot \beta_{S}(\mathcal{A}) \cdot e_{2,S}' \in \mathbb{Z}[G]$.
Then, as pointed out before, we always have numerically that $d \, | \, 2m$, whereas Burns's conjecture predicts only $d \, | \, 2m^{2}$, but we do not know of any theoretical reason that explains this phenomenon.  We shall distinguish four different statements:
\begin{enumerate}
\item $d \cdot \beta_{S}(\mathcal{A}) \cdot e_{2,S}' \in {\rm Ann}_{\mathbb{Z}[G]}(Cl(K))$, \label{un}
\item $2 \cdot m \cdot \beta_{S}(\mathcal{A}) \cdot e_{2,S}' \in {\rm Ann}_{\mathbb{Z}[G]}(Cl(K))$,  \label{deux}
\item $2 \cdot m^{2} \cdot \beta_{S}(\mathcal{A}) \cdot e_{2,S}' \in {\rm Ann}_{\mathbb{Z}[G]}(Cl(K))$, \label{trois}
\item $2 \cdot m^{2} \cdot \beta_{S}(\mathcal{A}) \cdot e_{2,S}' \in {\rm Ann}_{\mathbb{Z}[G]}(Cl_{S}(K))$.  \label{quatre}
\end{enumerate}
Under the assumption $d \, | \, 2m$,
note that (\ref{un}) implies (\ref{deux}) implies (\ref{trois}) implies (\ref{quatre}).  We list the number of them for each type of field $K$ (Galois abelian, Galois non-abelian and not Galois over $\mathbb{Q}$) in Tables \ref{table_three}, \ref{table_two} and \ref{table_four} below.  Part (\ref{un_11}) of Conjecture \ref{burnsquestion} is precisely the fourth statement.

Among our $19197$ examples, there are only $116$ examples where we have to go all the way to the fourth statement.  All of them satisfy $|S| = 3$ so there is only one finite ramified prime in those extensions.  Among these $116$ examples, there are only $2$ for which the $S_{K}$-class number is not $1$.  One of them is as follows.  

The base field is $k = \mathbb{Q}(\sqrt{42})$.  The rational prime $397$ splits completely in $k$.  Let $\mathfrak{p}$ be one of the two primes lying above $397$ and consider the ray class field $k_{\mathfrak{p}}$.  One has $[k_{\mathfrak{p}}:k] = 6$ and thus there is a unique subfield of degree $3$ over $k$ which we denote by $K$. A defining polynomial for $K$ is given by
$$p(x) = x^{6} - 2x^{5} - 61x^{4} + 84x^{3} + 708x^{2} - 640x - 1664 $$
and $K$ is not Galois over $\mathbb{Q}$.  The prime $\mathfrak{p}$ ramifies in $K/k$ and we let $\mathfrak{P}$ be the unique prime of $K$ lying above $\mathfrak{p}$.  Using PARI, we have $Cl(K) \simeq \mathbb{Z}/14\mathbb{Z}$.  We calculated an Artin system of $S_{K}$-units (which we do not list here), for which we have
$$d = 54782 = 2 \cdot 7^{2} \cdot 13 \cdot 43 $$ 
and
$$m = 191737 = 7^{3} \cdot 13 \cdot 43. $$
Note that in this case $d \nmid m$, but $d \, | \, 2m$.  Moreover, we have
$$2 \cdot m^{2} \cdot \beta_{S}(\mathcal{A}) \cdot e_{2,S}' = -1088490949\cdot {\rm id}+ 2645395389\cdot \sigma_{1}-1960894299\cdot \sigma_{2} \in \mathbb{Z}[G]. $$
Using PARI, we found an ideal $\mathfrak{a}$ such that $[\mathfrak{a}]$ generates $Cl(K)$.  If we let
$$\alpha = 2 \cdot m^{2} \cdot \beta_{S}(\mathcal{A}) \cdot e_{2,S}' \in \mathbb{Z}[G], $$
then $\mathfrak{a}^{\alpha}$ is not principal, but $\mathfrak{P}\cdot \mathfrak{a}^{\alpha}$ is.  So we do have 
$$\alpha \in {\rm Ann}_{\mathbb{Z}[G]}(Cl_{S}(K)) $$
as predicted by Burns's conjecture.

Finally, for those $116$ examples for which we have to go all the way to the fourth statement, we checked (\ref{burns_two}) of Conjecture \ref{burnsquestion} as follows: we let $S' = S_{\infty}$ and we pick
$$b \in \bigcup_{v \in S \smallsetminus S'}{\rm Ann}_{\mathbb{Z}[G]}\left(\mathbb{Z}[G/G_{v}] \right) $$
to be $b = \sigma -1$, where $\sigma$ is a non-trivial element of $G$.  In every single case, we verified that
$$(\sigma - 1) \cdot 2 \cdot m^{2} \cdot \beta_{S}(\mathcal{A}) \cdot e_{2,S}' \in {\rm Ann}_{\mathbb{Z}[G]}(Cl(K)). $$

As a final remark, in all our examples, not only does $d\,|\,2m$, but also
$$2\cdot m\cdot\beta_S(\mathcal{A})\cdot e'_{2,S}\in{\rm Ann}_{\mathbb{Z}[G]}(Cl_{S}(K)).$$
It might be of interest to investigate this further.

\section{Tables}\label{tables}

\begin{table}[h]
\begin{center} 
\caption{Summary of data.} 
\begin{tabular}{|c|c|c|c|c|c|c|c|c|c|c|c|c|c|c|c|c|c|c|c|}
\hline
Types $\setminus$ $h_{K}$ & $1$  &  $2$ & $3$ & $4$ & $5$ & $6$ & $7$ & $8$ & $9$  & $11$ & $12$ & $13$ & $14$ & $18$ & Total\\ 
\hline
\hline
Galois abelian & $478$  &  $245$ & $169$ & $91$ & $5$ & $37$ & $8$ & $8$ & $11$  & $0$ & $10$ & $1$ & $2$ &  $2$ & $1067$ \\
\hline
Galois  non-abelian & $218$  &  $81$ & $73$ & $17$ & $3$ & $21$ & $0$ & $3$ & $18$ & $0$ & $7$ & $0$ & $0$ &  $1$ & $442$ \\
\hline
Non Galois & $12340$  &  $1178$ & $3470$ & $268$ & $10$ & $196$ & $42$ & $6$ & $146$  & $4$ & $24$ & $0$ & $2$ &  $2$ & $17688$ \\
\hline
\hline
Total & $13036$  &  $1504$ & $3712$ & $376$ & $18$ & $254$ & $50$ & $17$ & $175$  & $4$ & $41$ & $1$ & $4$ &  $5$ & $19197$ \\
\hline
\end{tabular}   \label{table_one}
\end{center}
\end{table}

\begin{table}[h] 
\begin{center}
\caption{Annihilation statements for the Galois abelian case.} \
\begin{tabular}{|c|c|c|c|c|c|c|c|c|c|c|c|c|c|c|c|c|c|c|c|}
\hline
Statements $\setminus$ $h_{K}$ & $1$  &  $2$ & $3$ & $4$ & $5$ & $6$ & $7$ & $8$ & $9$  & $11$ & $12$ & $13$ & $14$ & $18$ & Total\\ 
\hline
\hline
Statement $1$ & $478$  &  $64$ & $165$ & $48$ & $4$ & $35$ & $5$ & $0$ & $9$ & $0$ & $10$ & $1$ & $0$ &  $1$ & $820$ \\
\hline
Statement $2$ & $0$  &  $122$ & $4$ & $39$ & $1$ & $1$ & $3$ & $7$ & $2$  & $0$ & $0$ & $0$ & $2$ &  $1$ & $182$ \\
\hline
Statement $3$ & $0$  &  $57$ & $0$ & $4$ & $0$ & $1$ & $0$ & $1$ & $0$  & $0$ & $0$ & $0$ & $0$ &  $0$ & $63$ \\
\hline
Statement $4$ & $0$  &  $2$ & $0$ & $0$ & $0$ & $0$ & $0$ & $0$ & $0$  & $0$ & $0$ & $0$ & $0$ &  $0$ & $2$ \\
\hline
\hline
Total & $478$  &  $245$ & $169$ & $91$ & $5$ & $37$ & $8$ & $8$ & $11$  & $0$ & $10$ & $1$ & $2$ &  $2$ & $1067$ \\
\hline
\end{tabular}  \label{table_three}
\end{center}
\end{table}

\begin{table}[h] 
\begin{center}
\caption{Annihilation statements for the Galois non-abelian case.} \
\begin{tabular}{|c|c|c|c|c|c|c|c|c|c|c|c|c|c|c|c|c|c|c|c|}
\hline
Statements $\setminus$ $h_{K}$ & $1$  &  $2$ & $3$ & $4$ & $5$ & $6$ & $7$ & $8$ & $9$  & $11$ & $12$ & $13$ & $14$ & $18$ & Total\\ 
\hline
\hline
Statement $1$ & $218$  &  $16$ & $73$ & $9$ & $2$ & $21$ & $0$ & $0$ & $17$ & $0$ & $4$ & $0$ & $0$ &  $1$ & $361$ \\
\hline
Statement $2$ & $0$  &  $54$ & $0$ & $8$ & $1$ & $0$ & $0$ & $3$ & $1$  & $0$ & $3$ & $0$ & $0$ &  $0$ & $70$ \\
\hline
Statement $3$ & $0$  &  $7$ & $0$ & $0$ & $0$ & $0$ & $0$ & $0$ & $0$  & $0$ & $0$ & $0$ & $0$ &  $0$ & $7$ \\
\hline
Statement $4$ & $0$  &  $4$ & $0$ & $0$ & $0$ & $0$ & $0$ & $0$ & $0$  & $0$ & $0$ & $0$ & $0$ &  $0$ & $4$ \\
\hline
\hline
Total & $218$  &  $81$ & $73$ & $17$ & $3$ & $21$ & $0$ & $3$ & $18$ & $0$ & $7$ & $0$ & $0$ &  $1$ & $442$ \\
\hline
\end{tabular}  \label{table_two}
\end{center}
\end{table}

\begin{table}[h] 
\begin{center}
\caption{Annihilation statements for the non Galois case.} \
\begin{tabular}{|c|c|c|c|c|c|c|c|c|c|c|c|c|c|c|c|c|c|c|c|}
\hline
Statements $\setminus$ $h_{K}$ & $1$  &  $2$ & $3$ & $4$ & $5$ & $6$ & $7$ & $8$ & $9$  & $11$ & $12$ & $13$ & $14$ & $18$ & Total\\ 
\hline
\hline
Statement $1$ & $12340$  &  $530$ & $3453$ & $179$ & $2$ & $186$ & $37$ & $0$ & $135$ & $0$ & $17$ & $0$ & $0$ &  $0$ & $16879$ \\
\hline
Statement $2$ & $0 $ & $320$  &  $9$ & $59$ & $0$ & $0$ & $1$ & $2$ & $11$   & $0$ & $7$ & $0$ & $0$ &  $2$ & $411$ \\
\hline
Statement $3$ & $0$  &  $254$ & $0$ & $18$ & $4$ & $6$ & $2$ & $4$ & $0$  & $0$ & $0$ & $0$ & $0$ &  $0$ & $288$ \\
\hline
Statement $4$ & $0$  &  $74$ & $8$ & $12$ & $4$ & $4$ & $2$ & $0$ & $0$  & $4$ & $0$ & $0$ & $2$ &  $0$ & $110$ \\
\hline
\hline
Total & $12340$  &  $1178$ & $3470$ & $268$ & $10$ & $196$ & $42$ & $6$ & $146$  & $4$ & $24$ & $0$ & $2$ &  $2$ & $17688$ \\
\hline
\end{tabular} \label{table_four}
\end{center}
\end{table}

\hfill

\bibliographystyle{plain}
\bibliography{main}

\end{document}